\documentclass[oneside,english]{amsart}
\usepackage[T1]{fontenc}
\usepackage[latin9]{inputenc}
\usepackage{color}
\usepackage{babel}
\usepackage{amsthm}
\usepackage{amssymb}
\usepackage{xargs}[2008/03/08]
\usepackage[unicode=true,pdfusetitle,
 bookmarks=true,bookmarksnumbered=false,bookmarksopen=false,
 breaklinks=false,pdfborder={0 0 1},backref=false,colorlinks=true]
 {hyperref}
\usepackage{breakurl}

\makeatletter
\numberwithin{equation}{section}
\numberwithin{figure}{section}
  \theoremstyle{remark}
  \newtheorem*{acknowledgement*}{\protect\acknowledgementname}
\theoremstyle{plain}
\newtheorem{thm}{\protect\theoremname}
  \theoremstyle{definition}
  \newtheorem{defn}[thm]{\protect\definitionname}
  \theoremstyle{remark}
  \newtheorem{rem}[thm]{\protect\remarkname}
  \theoremstyle{plain}
  \newtheorem{prop}[thm]{\protect\propositionname}
  \theoremstyle{plain}
  \newtheorem{lem}[thm]{\protect\lemmaname}
  \theoremstyle{plain}
  \newtheorem{cor}[thm]{\protect\corollaryname}
  \theoremstyle{definition}
  \newtheorem{example}[thm]{\protect\examplename}
  \theoremstyle{remark}
  \newtheorem*{rem*}{\protect\remarkname}

\usepackage{MnSymbol}
\usepackage{tikz}
\usetikzlibrary{positioning}
\usetikzlibrary{matrix,arrows}
\usepackage{rotating}

\DeclareMathOperator{\ccten}{\widehat{\boxtimes}}
\DeclareMathOperator{\from}{\leftarrow}

\makeatother

  \providecommand{\acknowledgementname}{Acknowledgement}
  \providecommand{\corollaryname}{Corollary}
  \providecommand{\definitionname}{Definition}
  \providecommand{\examplename}{Example}
  \providecommand{\lemmaname}{Lemma}
  \providecommand{\propositionname}{Proposition}
  \providecommand{\remarkname}{Remark}
\providecommand{\theoremname}{Theorem}

\begin{document}

\title{NONARCHIMEDEAN COALGEBRAS AND COADMISSIBLE MODULES}

\author{Anton Lyubinin}
\begin{abstract}
We show that basic notions of locally analytic representation theory
can be reformulated in the language of topological coalgebras (Hopf
algebras) and comodules. We introduce the notion of admissible comodule
and show that it corresponds to the notion of admissible representation
in the case of compact $p$-adic group.
\end{abstract}

\email{anton@ustc.edu.cn, anton@lyubinin.kiev.ua}

\address{Department of Mathematics, School of Mathematical Sciences, University
of Science and Technology of China, Hefei, Anhui, People's Republic
of China}

\maketitle
\global\long\def\Imm#1{\mathrm{Im}\left(#1\right)}
\global\long\def\Id{\mathrm{Id}}
\global\long\def\Ker#1{\mathrm{Ker}\left(#1\right)}

\global\long\def\Bann{\mathrm{Ban}_{K}^{\infty}}
\global\long\def\Banc{\mathrm{Ban}_{K}^{\leq1}}
\global\long\def\Ban{\mathrm{Ban}_{K}}

\global\long\def\BAlgn{\mathrm{BAlg}_{K}^{\infty}}
\global\long\def\BAlgc{\mathrm{BAlg}_{K}^{\leq1}}
\global\long\def\BAlg{\mathrm{BAlg}_{K}}

\global\long\def\BCAlgn{\mathrm{BCoalg}_{K}^{\infty}}
\global\long\def\BCAlgc{\mathrm{BCoalg}_{K}^{\leq1}}
\global\long\def\BCAlg{\mathrm{BCoalg}_{K}}

\global\long\def\Bmod#1{\mathrm{BMod}_{#1}}
\global\long\def\Bmodc{\mathrm{Bmod}_{K}^{\leq1}}

\global\long\def\rctcomod#1{\mathrm{CTComod}_{#1}}
\global\long\def\lctcomod#1{\mathrm{_{#1}CTComod}}

\global\long\def\rlscomod#1{\mathrm{LSComod}_{#1}}
\global\long\def\llscomod#1{\mathrm{_{#1}LSComod}}

\global\long\def\rdlsmod#1{\mathrm{FSMod}_{#1}}
\global\long\def\ldlsmod#1{\mathrm{_{#1}FSMod}}

\global\long\def\rnfmod#1{\mathrm{NFMod}_{#1}}
\global\long\def\lnfmod#1{\mathrm{_{#1}NFMod}}

\global\long\def\rbcom#1{\mathrm{BComod}_{#1}}
\global\long\def\rbcomn#1{\mathrm{BComod}_{#1}^{\infty}}
\global\long\def\rbcomc#1{\mathrm{BComod}_{#1}^{\leq1}}

\global\long\def\lbcom#1{\!{}_{#1}\mathrm{BComod}}
\global\long\def\lbcomn#1{\!{}_{#1}\mathrm{BComod}^{\infty}}
\global\long\def\lbcomc#1{\!{}_{#1}\mathrm{BComod}^{\leq1}}

\global\long\def\rbmod#1{\mathrm{BMod}_{#1}}
\global\long\def\lbmod#1{\!_{#1}\mathrm{BMod}}

\global\long\def\uball#1{B_{#1}}

\global\long\def\sd#1{\left(#1\right)'_{b}}
\global\long\def\wdual#1{\left(#1\right)'_{s}}
\global\long\def\bd#1{#1'}
\global\long\def\bhd#1{#1^{\odot}}

\global\long\def\hits{\rightharpoonup}
\global\long\def\hit{\leftharpoonup}

\global\long\def\cten{\widehat{\otimes}}
\global\long\def\ten{\otimes}
\global\long\def\ep{\epsilon}
\global\long\def\emb{\hookrightarrow}
\global\long\def\dperp{\upmodels}
\global\long\def\from{\leftarrow}

\global\long\def\cm{\Delta}
\global\long\def\bten{\bar{\otimes}}
 \newcommandx\norm[1][usedefault, addprefix=\global, 1=\cdot]{\left\Vert #1\right\Vert }
\global\long\def\ccten#1{\underset{#1}{\widehat{\boxtimes}}}
\global\long\def\hd#1{#1^{\widehat{0}}}

\global\long\def\id#1{\mathrm{id}_{#1}}
\global\long\def\ann#1{\left(#1\right)^{\perp}}

\global\long\def\ilim{{\displaystyle \lim_{\longrightarrow}\,}}
\global\long\def\plim{{\displaystyle \lim_{\longleftarrow}\,}}

\newcommandx\compl[2][usedefault, addprefix=\global, 1=]{#2^{\wedge#1}}

\tableofcontents{}

\section*{Introduction}

The study of $p$-adic locally analytic representation theory of $p$-adic
groups seems to start in 1980s, with the first examples of such representations
studied in the works of Y. Morita \cite{M1,M2,M3} (and A. Robert,
around the same time), who considered locally analytic principal series
representations for $p$-adic $SL_{2}$. Morita determined when those
representations are irreducible by analyzing them as topological modules
over the corresponding Lie algebra with the action of a finite group.
Among the other work on $p$-adic (non-analytic) representations,
relevant to this paper, we would like to mention the works of B. Diarra,
who was focusing on special classes of representations of compact
groups. Namely, Diarra studied nonarhimedean Banach Hopf algebras
and his major applications were to the classes of representations,
whose Hopf algebra of matrix elements admit an invariant functional.
Diarra also considered Banach Hopf algebra $C\left(\mathbb{Z}_{p},K\right)$
of continuous functions on $p$-adic integers $\mathbb{Z}_{p}$ and
showed that the category of continuous representations of $\mathbb{Z}_{p}$
in $p$-adic Banach spaces is equivalent to the category of Banach
comodules over $C\left(\mathbb{Z}_{p},K\right)$. Diarra also constructed
examples of irreducible infinite-dimensional representations of $\mathbb{Z}_{p}$,
thus showing that even in the case of compact group, the picture over
$\mathbb{Q}_{p}$ is a lot more complicated than the one over $\mathbb{C}$.

The major development of the theory was made in a series of papers
of P. Schneider and J. Teitelbaum\cite{ST1,ST2,ST3,ST4}. Their approach
to continuous or locally analytic representations of a locally analytic
group $G$ is algebraic and, instead of analyzing representation $V$
as topological $K$-vector space, they work with its strong dual space
$\sd V$ as with a module over distribution algebra $D^{la}\left(G,K\right)$
($D\left(G,K\right)$ for continuous representations). They were able
to obtain results similar to Morita's for principal series representations
of $p$-adic $GL_{2}$ and formulated good finiteness conditions on
representations, called admissibility. Admissible representations
rule out pathologies of $p$-adic case, like infinite-dimensional
irreducible representations of $\mathbb{Z}_{p}$, and form a broad
enough category of representations, which include all important examples
and provide a framework for developing a general theory. By definition,
a representation is admissible if its strong dual is a coadmissible
module over $D^{la}\left(G,K\right)$ and the later means it is a
coadmissible module over $D^{la}\left(H,K\right)$ for some (and thus
any) open compact subgroup $H<G$. For a compact subgroup $H<G$ the
algebra $A=D^{la}\left(H,K\right)$ have an additional structure,
called Fréchet-Stein structure, which means it is a locally convex
projective limit $A=\plim A_{n}$ of Noetherian Banach algebras $A_{n}$
with flat transition maps. In \cite{ST3} Schneider and Teitelbaum
prove basic results for general Fréchet-Stein algebras and define
coadmissible modules w.r.t. Fréchet-Stein structure $\left\{ A_{n}\right\} $
as modules that are projective limits $M=\plim M_{n}$ of finitely-generated
$A_{n}$-modules $M_{n}$, such that we also have isomorphisms $M_{n}\cong A_{n}\otimes_{A_{n+1}}M_{n+1}$.
They prove that $D^{la}\left(H,K\right)$ is a Fréchet-Stein algebra
which is additionally is a nuclear space. In \cite{EM} the notion
of coadmissible module was generalized to the case of weak Fréchet-Stein
algebras, which can be thought of as projective limits of Banach algebras.

In this paper we develop a corresponding notion for comodules over
certain topological coalgebras. Namely, we define the category of
compact type (CT) $\cten$-coalgebras and show that this category
is antiequivalent to the category of nuclear Fréchet (NF) algebras
(however, our terminology is a little different from \cite{ST3} and
\cite{EM}). We give a definition of an admissible comodule over a
CT-$\cten$-coalgebra and of a coadmissible module over an NF-$\cten$-algebra
and we show that categories of admissible comodules and coadmissible
modules are antiequivalent. In case of CT-$\cten$-coalgebra $C$
with $\sd C$ being a nuclear Fréchet-Stein algebra, our definition
of coadmissible module over $\sd C$ is equivalent to the one of Schneider
and Teitelbaum. Thus for a compact $p$-adic group $H$ our admissible
comodules over CT-$\cten$-coalgebra $C^{la}\left(H,K\right)$ provide
an explicit description of the category of admissible representations
of $H$, since the category of locally analytic representations of
$H$ in CT-spaces is equivalent to the category of CT-$\cten$-comodules
over $C^{la}\left(H,K\right)$. Hopf algebra and comodule formalism
has proven to be useful in representation theory and provide a natural
language for treatment of certain questions. In particular, it is
the language of quantum groups (for example, see \cite{PW}) and many statements
of the geometric representation theory of algebraic groups are formulated
and proven in quantum case. Besides applications to quantization (see
\cite{S}), which was one of the original motivations for this work, it is
our belief that a similar thing can be done in $p$-adic representation
theory and thus Hopf algebra formalism can be an alternative to geometric
methods. Finally, comodule 
and Hopf algebra formalism provide the most convenient framework for discussing 
Tannaka reconstruction (the main motivation for this work, with nonarchimedean version being the subject of \cite{Lyu2}), which, in its turn, is the key element in the proof of the 
geometric Satake correspondence. This paper, and most importantly the description of admissible
representations of a compact $p$-adic group in terms of comodules,
provides the foundation for further work in all these directions. 

Let us outline the content of this paper.

In the first two sections we develop the basic notions of the theory
of Banach $\cten$-coalgebras (bialgebras, Hopf algebras). Some of
the results there have been worked out in \cite{bd1,bd2,bd3,bd4,BD5}
for Hopf $\cten$-algebras, some are more or less straightforward
generalizations of the corresponding notions and statements of the
algebraic theory of coalgebras and comodules \cite{dnr,sw}. We also
formulate some concepts and prove results, foundational for the theory
of comodules, like Frobenius reciprocity and tensor identity, which
will not be used along the proof of our main result. Some of the traditional
questions of Hopf algebra theory, like the description of rational
modules (\cite[2.1]{sw}, for more general version see \cite{L})
have very simple answers in $p$-adic case. Since our main result
is a ``module-comodule'' duality, we focus our attention on coalgebras
and most of the time are sketchy on bialgebra and Hopf algebra case.

The last three sections are devoted to the topological theory. Although
some basic results share the same proofs in Banach and more general
topological case, we think it is better to consider those two cases
separately for the sake of clarity. Our main objects of study are
a $\cten$-coalgebras $C$, such that $C$ is topologically isomorphic
to a compact locally convex inductive limit of a sequence of Banach
$\cten$-coalgebras $C_{n}$, and a $\cten$-comodules $V$, such
that $V$ is topologically isomorphic to a compact locally convex
inductive limit of a sequence of Banach $C_{n}-\cten$-comodules $V_{n}$.
We call such objects CT-coalgebras and CT-comodules correspondingly.
However, these notions are not sufficient to formulate our notion
of admissibility for CT-comodule, for which one need to consider $\cten$-comodules
of the form $V\ccten CC_{n}$. The cotensor product $V\ccten CC_{n}$
is a closed subspace of the complete tensor product $V\cten C_{n}$.
The problem is that $V\cten C_{n}$ is not a compact type space, but
only an LB-space, (not compact) locally convex inductive limit of
Banach spaces, and $V\ccten CC_{n}$ does not even have to be an LB-space.
This forces us to consider more general topological comodules, which
we call LB- and LS-comodules. Since CT-spaces are also called LS-spaces
in functional analysis, we use this name for topological algebraic
structures, which are compact type only as a space, i.e., LS-comodule
$V$ is a topological comodule, such that $V\cong\ilim V_{n}$ is
a compact locally convex inductive limit of Banach spaces $V_{n}$,
but $V_{n}$ are only spaces, not comodules themselves.

Since our intended audience has various background (Hopf algebras,
quantum groups and $p$-adic representation theory), we moved proofs
of two key technical results from functional analysis into appendix.
We hope it makes this paper a better reading.
\begin{acknowledgement*}
We would like to thank an anonymous referee for pointing out the reference
\cite{BD5} and B. Diarra for sharing this old paper with us. We also
would like to thank W. Schikhof and C. Perez-Garcia for answering
some questions of the author. Finally, we would like to thank Z. Lin
for his remarks on early versions of this paper.
\end{acknowledgement*}

\section{Banach coalgebras}

Throughout this paper $K$ means a discretely valued field with norm
$\left|\cdot\right|_{K}$. $K^{0}:=\left\{ x\in K|\,\left|x\right|_{K}\leq1\right\} $
is its ring of integers, $K^{00}:=\left\{ x\in K|\,\left|x\right|_{K}<1\right\} $
and $\bar{K}:=K^{0}/K^{00}$ is the residue field. The linear dual
space of a linear space $V$ we denote as $V^{*}$. For a linear map
$A:V\to W$ of vector spaces we denote by $A^{*}$ its adjoint map. 

A \emph{Banach space over} $K$ (or $K$-Banach space) is a normed
space $\left(V,\norm_{V}\right)$ such that $V$ is complete w.r.t.
$\norm_{V}$. When no confusion can occur, we will just write $V$.
$K$-Banach spaces form a category $\Ban$ with bounded linear maps
as morphisms. For $V,W\in\Ban$, $\Ban\left(V,W\right)$ is the Banach
space itself and we write $C\left(V,W\right)\subset\Ban\left(V,W\right)$
for the closed subspace of compact maps. $V^{0}:=\left\{ x\in V|\,\norm[x]_{V}\leq1\right\} $
is a closed unit ball in $V$, $V^{00}:=\left\{ x\in V|\,\norm[x]_{V}<1\right\} $
and $\bar{V}:=V^{0}/V^{00}$. $\bar{V}$ is a vector space over $\bar{K}$.

For convenience we always suppose that the norm $\norm_{V}$ is \emph{solid},
i.e. $\norm[V]_{V}\subset\left|K\right|_{K}$. From \cite[2.1.9]{PGSch}
it is known that this assumption is not a restriction.

We will say that a $K$-Banach space $V$ is \emph{free} if it is
of the form $c_{0}\left(E_{V},K\right)=\left\{ \sum_{n=0}^{\infty}a_{n}e_{n}|\, e_{n}\in E_{V},\, a_{n}\in K:\, a_{n}\to0\right\} $.
Over discretely valued fields any Banach space $V$ is free, with
a possible choice of $E_{V}$ being a $\bar{K}$-vector space basis
of $\bar{V}$.

The continuous dual of a $K$-Banach space $V$ will be denoted by
$\bd V$. The adjoint map for a continuous linear map $f$ will also
be denoted by $f'$.

For a subset $U\subset V$ its annihilator is denoted by $U^{\perp}=\left\{ \phi\in\bd V:\,\phi\left(x\right)=0\;\forall x\in U\right\} $.
The completion of a linear (sub)space $U$ w.r.t. a given norm (or
topology $\tau$) will be denoted by $\compl U$ ($\compl[\tau]U$).

For two $K$-Banach spaces $V$ and $W$ one defines a norm $\norm_{V\cten W}$
on the space $V\otimes_{K}W$ as $\norm[u]_{V\cten W}=\inf\left\{ \sup\norm[v_{i}]_{V}\norm[w_{i}]_{W}|\mbox{ }u=\sum v_{i}\otimes w_{i}\right\} $,
where infinum is taken over all decompositions of $u=\sum v_{i}\otimes w_{i}$.
We denote the completion of the space $V\otimes_{K}W$ w.r.t. $\norm_{V\cten W}$
by $V\cten W$ and call it the complete(d) tensor product of $V$
and $W$. It is known that the topology, induced by $\norm_{V\cten W}$
on $V\cten W$ coincide with projective and inductive tensor product
topologies \cite[section 17]{nfa}. $\Ban$ has a symmetric monoidal
category structure w.r.t. $\cten$. We also use the notation $\bar{\ten}$
to denote the composition of the tensor product ($\cten$ or $\ten$)
with the canonical isomorphism $K\cten V\cong V$ ($V\cten K\cong V$),
i.e. we write $a\bar{\ten}b$ when either $a$ or $b$ belongs to
$K$ and for maps $f\bar{\ten}g$ when the range of either $f$ or
$g$ is $K$.

In a few places we use Sweedler's notation, so we briefly recall it.
For comultiplication on a coassociative coalgebra $C$, $\Delta\left(c\right)={\displaystyle \sum_{i}}a_{i}\ten b_{i}$
in order to avoid introducing new letters one writes $\Delta\left(c\right)={\displaystyle \sum_{i}}c_{\left(1\right)i}\ten c_{\left(2\right)i}$.
In Sweedler's notation one does not think of the summation symbol
$i$, so we just write $\Delta\left(c\right)={\displaystyle \sum}c_{\left(1\right)}\ten c_{\left(2\right)}$
or, in sumless notation, simply $\Delta\left(c\right)=c_{\left(1\right)}\ten c_{\left(2\right)}$.
The coassociativity $\left(\cm\ten\id{}\right)\circ\cm=\left(\id{}\ten\cm\right)\circ\cm$
of the comultiplication $\cm$ allows us to write $c_{\left(1\right)}\ten c_{\left(2\right)}\ten c_{\left(3\right)}:=c_{\left(1\right)\left(1\right)}\ten c_{\left(1\right)\left(2\right)}\ten c_{\left(2\right)}=c_{\left(1\right)}\ten c_{\left(2\right)\left(1\right)}\ten c_{\left(2\right)\left(2\right)}$.
One can also use the notation for coactions $\rho_{V}:V\to V\ten C$,
$\rho_{V}\left(v\right)=v_{\left(0\right)}\ten v_{\left(1\right)}$.
In our setting (Banach or locally convex) the sum in Sweedler's notation
is always assumed to be countable and convergent.

In this section we provide a summary of basic facts about Banach coalgebras
(Hopf algebras). Some of them were worked out in \cite{bd1,bd2,bd3,bd4,BD5}
for Hopf algebras.

\subsection{Banach $\cten$-Coalgebras}
\begin{defn}
A (coassociative and counital)\emph{ $K$-Banach $\cten$-coalgebra}
(or simply Banach coalgebra) $C$ is a triple $\left(C,\Delta_{C},\epsilon_{C}\right)$,
where $C$ is a $K$-Banach space, $\Delta_{C}:C\to C\cten C$ is
a \emph{comultiplication} and $\epsilon_{C}:C\to K$ is a \emph{counit},
continuous morphisms, such that $\left(C,\Delta_{C},\epsilon_{C}\right)$
is a coalgebra object in $\Ban$.
\end{defn}
Taking continuous duals give us the map 
\[
\cm_{C}':\bd{\left(C\cten C\right)}\to\bd C.
\]
 Since we have an embedding $\bd C\cten\bd C\emb\bd{\left(C\cten C\right)}$,
we get on $C'$ the structure of a Banach $\cten$-algebra, with the
product given by 
\[
\alpha\star\beta=\cm_{C}'\circ\left(\alpha\cten\beta\right)=\left(\alpha\bten\beta\right)\circ\cm_{C}.
\]
The adjoint map of counit $\ep_{C}':K\to\bd C$ is the unit of this
algebra.
\begin{defn}
$\bd C$ is called \emph{dual Banach }$\cten$-\emph{algebra} for
$C$. The $\left(-\star-\right)$ product is called \emph{convolution}
product.
\end{defn}
\smallskip{}

\begin{rem}
If we take $C$ - finite-dimensional, then $\bd{\left(C\cten C\right)}=\bd C\cten\bd C,$
and thus, in case of Hopf algebra $C$, $\bd C$ also becomes a Hopf
algebra by standard argument.
\begin{defn}
Let $C$ and $B$ be Banach $\cten$-coalgebras. A bounded linear
map $f:C\to B$ is a \emph{(counital) morphism of Banach }$\cten$-\emph{coalgebras}
if 
\[
\cm_{B}\circ f=\left(f\cten f\right)\circ\cm_{C}
\]
 and 
\[
\epsilon_{B}\circ f=\epsilon_{C}
\]

\end{defn}
\end{rem}
The dual map $f':\bd B\to\bd C$ would satisfy 
\[
f'\circ\cm_{B}'=\cm_{C}'\circ\left(f'\cten f'\right)
\]
 and 
\[
f'\circ\epsilon_{B}'=\epsilon_{C}'
\]
i.e. $f'$ is a (unital) homomorphism of the Banach $\cten$-algebras
$\bd B$ and $\bd C.$ 

Denote the category of Banach $\cten$-coalgebras over $K$ by $\BCAlg$
and the category of Banach $\cten$-algebras over $K$ by $\BAlg$.
For $C,B\in\BCAlg$, $\BCAlg\left(C,B\right)$ is a {*}-weakly closed
subset of $\Ban$$\left(C,B\right)$. A similar statement is true
for homomorphisms of Banach $\cten$-algebras.

Again, in finite-dimensional setting, the dual of the $\cten$-algebra
homomorphism is a $\cten$-coalgebra morphism. 

\smallskip{}

\begin{rem}
Since for a $K$-Banach space $C$ we have $C\otimes_{K}C\subset C\cten C$,
Banach $\cten$-algebras form a subcategory of usual associative algebras.
This is no longer true for $\cten$-coalgebras. Neither Banach $\cten$-coalgebras
are a subcategory of coassociative coalgebras, nor coassociative coalgebras
form a subcategory of Banach $\cten$-coalgebras. For the first statement,
the reason is that comultiplication on a Banach $\cten$-coalgebra
$C$ acts from $C$ into the completed tensor product $C\cten C$,
which is larger than the algebraic tensor product $C\otimes_{K}C$.
For the second, the comultiplication on a coassociative coalgebra
may not be continuous.
\end{rem}

\subsection{Constructions in the category of Banach $\cten$-coalgebras}
\begin{defn}
Let $A$ be a subspace of $C$. $A$ is a $\cten$-\emph{subcoalgebra},
if $\cm_{C}\left(A\right)\subseteq A\cten A$. In case $A$ is closed,
we will say that $A$ is a closed (or Banach) $\cten$-subcoalgebra.

We will say that Banach $\cten$-coalgebra $C$ is \emph{simple} if
it does not have a proper nonzero closed $\cten$-subcoalgebra.
\end{defn}
Let $S\subset C$ be a subset in a Banach $\cten$-coalgebra $C$.
Denote by $C_{S}$ the intersection of all closed $\cten$-subcoalgebras
containing $S.$ Although the proof uses some notions that will be
given later, we formulate the following fact.
\begin{prop}
$\forall c\in C$: $C_{c}$ is a Banach space of countable type. Thus
any simple $\cten$-coalgebra is of countable type.\end{prop}
\begin{proof}
Similar to \cite[2.2.1]{sw}, see also \cite[I.2.1.ii]{bd2}.
\end{proof}
Let $\left(C,\norm[\cdot]_{C}\right)$ be a normed vector space over
$K$. Then the norm $\norm[\cdot]_{C}$ induces the norm $\norm[\cdot]_{C\cten C}$
on $C\otimes C$. Let $C$ be a coassociative counital coalgebra over
$K$, such that the coaction $\cm_{C}$ and counit $\epsilon_{C}$
are continuous with respect to the norm $\norm[\cdot]_{C}$ on $C$
and corresponding norm $\norm[\cdot]_{C\cten C}$ on $C\otimes C$.
Then for the completion $\compl C$ of $C$, from the universal property
of completion and continuity of $\cm_{C},$ one can extend the coaction
and counit and thus get a $K$-Banach coalgebra $\left(\compl C,\cm_{\compl C},\epsilon_{\compl C}\right)$.

\subsubsection{\label{sub:Dual-coalgebra}Dual coalgebra}

We review the notion of the dual Banach coalgebra from \cite{BD5}.

Let $A$ be a unital Banach algebra, $m_{A}:A\cten A\to A$ is the
multiplication and $\norm[m_{A}]\leq1$. Let $\bd A:=\sd A$ be the
dual Banach space and $m_{A}':\bd A\to\bd{\left(A\cten A\right)}$
be the dual map, which is necessarily isometric. We have an isometric
embedding $\bd A\cten\bd A\emb\bd{\left(A\cten A\right)}$ \cite[4.34]{vanr}.

Let $\bhd A:=\left(m_{A}'\right)^{-1}\left(\bd A\cten\bd A\right)$.
Recall that $\bd A$ is an $A$-bimodule with the left action $\left(a\hits a'\right)\left(x\right)=a'\left(xa\right)$
and the right action $\left(a'\hit a\right)\left(x\right)=a'\left(ax\right)$
for $a'\in\bd A$, $a\in A$ ($\gamma_{a'}\left(a\right)$ for $a\hits a'$
and $\delta_{a'}\left(a\right)$ for $a'\hit a$ in \cite{BD5}).
For a fixed $a'$ these define a continuous linear maps $\hits a',a'\hit:A\to\bd A$
($\gamma_{a'}$ for $\hits a'$ and $\delta_{a'}$ for $a'\hit$ in
\cite{BD5}), $\hits a'\left(a\right)=a\hits a'$ and $a'\hit\left(a\right)=a'\hit a$,
with $\norm[\hit a']=\norm[\hits a']=\norm[a']_{\bd A}$. 
\begin{prop}
\label{propBHD}Let $a'\in\bhd A$. For $u:A\to B$ a homomorphism
of unital Banach algebras, denote by $\bhd u$ the restriction $u'|_{\bhd B}$.
\begin{enumerate}
\item $\hits a',\, a'\hit:A\to\bd A$ are completely continuous \emph{(}compact\emph{)}
operators \cite[lemma 1]{BD5}.
\item $a\hits a'\in\bhd A$ and $a'\hit a\in\bhd A$ for all $a\in A$,
and thus $\hits a',\, a'\hit:A\to\bhd A$ \cite[lemma 2]{BD5}.
\item $m_{A}'\left(\bhd A\right)\subset\bhd A\cten\bhd A$.\cite[Theorem 1]{BD5}.
\item If $u:A\to B$ is a homomorphism of unital Banach algebras. Then $\bhd u\left(\bhd B\right)\subset\bhd A$
\cite[lemma 3]{BD5}. Thus we have a functor $\bhd{\left(-\right)}$.
\item For two unital Banach algebras $A$ and $B$ we have $\bhd{\left(A\cten B\right)}=\bhd A\cten\bhd B$
\cite[Theorem 2]{BD5}.
\item For a Banach algebra $\left(A,m_{A},u_{A}\right)$ its dual $\left(\bhd A,\bhd{m_{A}},\bhd{u_{A}}\right)$
is a Banach coalgebra \cite[Theorem 3]{BD5}. 
\end{enumerate}
\end{prop}
\begin{rem}
For a Banach Hopf algebra $\left(H,m_{H},u_{H},\cm_{H},\ep_{H}\right)$
(see section \ref{sub:Banach--bialgebras-and}) one get's a dual Banach
Hopf algebra $\left(\bhd H,\bhd{\cm_{H}},\bhd{\ep_{H}},\bhd{m_{H}},\bhd{u_{H}}\right)$
\cite[Theorem 4]{BD5}.\end{rem}
\begin{defn}
We call $\bhd A$ the dual Banach coalgebra of $A$.
\end{defn}
If $A=\bd C$ for some Banach coalgebra $C$, then, since $K$ is
discretely valued field, $C$ is pseudo-reflexive Banach space and
thus $C\emb\bd{\bd C}$ is a strict injection. Similar to \cite[1.5.12]{dnr},
one can show that $C\emb\bhd{\bd C}$ is an embedding of Banach coalgebras.

Clearly, if $u:A\to B$ is a homomorphism of unital Banach algebras
then $\bhd u$ is a Banach coalgebra morphism (or see \cite[1.5.4]{dnr}).
Thus we have a functor $\bhd{\left(-\right)}:\BAlg\to\BCAlg$ \cite[1.5.5]{dnr}.
Similar to \cite[1.5.22]{dnr}, one can show that $\bhd{\left(-\right)}$
is the left adjoint functor to the dual functor $\bd{\left(-\right)}:\BCAlg\to\BAlg$.

One can also define an analog of Sweedler\textquoteright{}s finite
dual of an algebra. Recall \cite[sec. 1.5]{dnr} that for any $K$-algebra
$A$ one has a notion of \emph{finite dual} $A^{0}$. It is a set
of linear functionals $f\in Hom_{K}\left(A,K\right)$, such that $\Ker f$
contains an ideal of finite codimension. $A^{0}$ is a coalgebra and
if $A$ is a bialgebra (Hopf algebra) then $A^{0}$ is also a bialgebra
(Hopf algebra). 

In Banach case one can consider the set $A^{0*}=A^{0}\cap\bd A$.
$A^{0*}$ can be described as a set of continuous functionals, who's
kernel contains a closed ideal of finite codimension. Also denote
by $A^{\widehat{0}}$ the topological closure of $A^{0*}$ in $\bd A.$
\begin{prop}
\label{prop:Properties-of-cfin}Properties of $A^{0}$ imply the following
properties of $A^{\widehat{0}}$:
\begin{enumerate}
\item If $f:A\to B$ is a continuous morphism of Banach $\cten$-algebras
and $I$ is a closed ideal of finite codimension in $B$, then $f^{-1}\left(I\right)$
is a closed finite-codimensional $\cten$-ideal \cite[1.5.1]{dnr};
\item $f^{*}\left(B^{\widehat{0}}\right)\subset A^{\widehat{0}}$\cite[1.5.2.i]{dnr};
\item if $\phi:\bd A\cten\bd B\to\bd{\left(A\cten B\right)}$ is a canonical
(isometric) embedding, then we have $\phi\left(A^{0*}\otimes B^{0*}\right)=\left(A\cten B\right)^{0*}$
and $\phi\left(\hd A\cten\hd B\right)=\hd{\left(A\cten B\right)}$
\cite[1.5.2.ii]{dnr};
\item if $m:A\cten A\to A$ is a multiplication on $A$, then $m^{*}\left(A^{0*}\right)\subset A^{0*}\otimes A^{0*}$
and $m^{*}\left(A^{\widehat{0}}\right)\subset A^{\widehat{0}}\cten A^{\widehat{0}}$
(same as in \cite[1.5.2.iii]{dnr});
\item $\left(A^{0*},m^{*}\circ\phi^{-1},\epsilon\right)$ is a coalgebra
and $\left(A^{\widehat{0}},m^{*}\circ\phi^{-1},\epsilon\right)$ is
a Banach $\cten$-coalgebra, with $\epsilon\left(a^{*}\right)=a^{*}\left(1\right)$
(same as in \cite[1.5.3]{dnr});
\item \label{enu:hdmorph}Let $f:A\to B$ be a Banach $\cten$-algebra morphism.
Then $f^{0*}=f^{*}|_{B^{0*}}:B^{0*}\to A^{0*}$ is a coassociative
coalgebra morphism and $f^{\widehat{0}}=f^{*}|_{B^{\widehat{0}}}:B^{\widehat{0}}\to A^{\widehat{0}}$
is a Banach $\cten$-coalgebra morphism (\cite[1.5.4]{dnr}).
\end{enumerate}
\end{prop}
Thus we have a $A^{0*}$ is a coassociative, counital coalgebra and
$\hd A$ is its completion, which we will call \emph{complete finite
}(\emph{c-finite }for short)\emph{ dual}. For a map of Banach $\cten$-algebras
$f:A\to B$ we define $\hd f:\hd B\to\hd A$ as the restriction of
the dual map $f':\bd B\to\bd A$. Thus we have a contravariant functor
$\left(-\right)^{\widehat{0}}:\BAlg\to\BCAlg$.

One has embeddings $A^{0*}\subset\hd A\subset\bhd A$. If $A=\bd C$
for some Banach coalgebra $C$, then $C\subset\bd C^{0*}$ only if
$C$ is a completion of a coassociative coalgebra. In general one
only has $C\subset\bhd{\bd C}$ and the general relation between $\hd A$
and $\bhd A$ is currently unknown.

We have the following description 
\[
\bhd A=\left\{ f\in\bd A|\mbox{ }\exists f_{i},g_{i}\in\bd A,i\in\mathbb{N}:\mbox{ }f\left(xy\right)=\sum_{i=1}^{\infty}f_{i}\left(x\right)g_{i}\left(y\right)\mbox{ }\forall x,y\in A\right\} .
\]

\subsubsection{$\cten$-coideals}
\begin{defn}
Let $C$ be a Banach $\cten$-coalgebra and $V$ be a closed subspace
of $C$.
\begin{itemize}
\item $V$ is called \emph{right closed }$\cten$-\emph{coideal} if we have
$\cm\left(V\right)\emb V\cten C$;
\item $V$ is called \emph{left closed }$\cten$-\emph{coideal} if we have
$\cm\left(V\right)\emb C\cten V$;
\item $V$ is called a (two-sided) \emph{closed }$\cten$-\emph{coideal}
if we have $\cm\left(V\right)\emb V\cten C+C\cten V$ and $\epsilon\left(V\right)=0$
\end{itemize}
\end{defn}
Completed sums and intersections of (left, right, two-sided) (closed)
$\cten$-coideals is again a (left, right, two-sided) closed $\cten$-coideal.

If $V$is a closed $\cten$-coideal in a Banach $\cten$-coalgebra
$C$, one can consider the projection map $\pi:C\to C/V$ in $\Ban$.
\begin{prop}
\cite[1.4.7]{sw}
\begin{enumerate}
\item $C/V$ is a Banach $\cten$-coalgebra and $\pi$ is a Banach $\cten$-coalgebra
map;
\item if $f:C\to D$ - Banach $\cten$-coalgebra map, then $\Ker f$ is
a closed $\cten$-coideal;
\item if f is surjective, than there is a canonical topological isomorphism
$C/\Ker f\cong D$;
\item universal property of quotient holds
\end{enumerate}
\end{prop}
\smallskip{}

\begin{prop}
Let $f\in\BCAlg\left(C,D\right)$. Then
\begin{enumerate}
\item Preimage of a closed $\cten$-coideal in $D$ is a closed $\cten$-coideal
in $C$;
\item Closure of the image of $f$ is a Banach $\cten$-subcoalgebra of
$D$.
\end{enumerate}
\end{prop}
\begin{proof}
1. Consider a $\cten$-coideal $I$ in $D$. Then the kernel of the
composition $\pi\circ f:C\to D/I$ is exactly $f^{-1}\left(I\right)$.

2. Since $f$ is a morphism of $\cten$-coalgebras, we have $\cm_{D}\left(f\left(C\right)\right)\subset f\left(C\right)\cten f\left(C\right)$.
Since we have $f\left(C\right)\cten f\left(C\right)=\compl{f\left(C\right)}\cten\compl{f\left(C\right)}$
and since $\cm_{D}$ is continuous, we have $\cm_{D}\left(\compl{f\left(C\right)}\right)\subset\compl{f\left(C\right)}\cten\compl{f\left(C\right)}$.
\end{proof}

\subsection{\label{sub:Banach--bialgebras-and}Banach $\cten$-bialgebras and
Hopf $\cten$-algebras}
\begin{defn}
A Banach $\cten$-bialgebra is a tuple $\left(B,m_{B},u_{B},\cm_{B},\epsilon_{B}\right)$
that is a bialgebra object in $\Ban$.
\end{defn}
A map is a \emph{morphism of Banach }$\cten$-\emph{bialgebras} if
it is a Banach $\cten$-algebra and Banach $\cten$-coalgebra homomorphism. 
\begin{rem}
As in the algebraic case, there is no $\left\{ 0\right\} $ $\cten$-bialgebra!
(since in this case $\epsilon\left(1\right)\neq1$)
\end{rem}
\smallskip{}

\begin{rem}
If $B$ is a finite-dimensional $\cten$-bialgebra., then $\bd B$
is a $\cten$-bialgebra.
\end{rem}
Recall that in general morphism sets in $\BAlg$ and $\BCAlg$ don't
have additional algebraic structure. For any Banach $\cten$-bialgebra
$B$ the space of endomorphisms $\mbox{Ban}_{K}\left(B\right)$ is
an algebra with respect to the convolution product 
\[
\forall\phi,\psi\in\mbox{Ban}_{K}\left(B\right):\mbox{ }\phi\star\psi:=m_{B}\circ\left(\phi\cten\psi\right)\circ\cm_{B}.
\]
It is clear that the convolution is a bounded map, since $\left\Vert \phi\star\psi\right\Vert \leq\norm[m_{B}]\left\Vert \phi\right\Vert \left\Vert \psi\right\Vert \left\Vert \cm_{B}\right\Vert ,$
and thus $\mbox{Ban}_{K}\left(B\right)$ is a $K$-Banach algebra
w.r.t. convolution, with the unit $u_{B}\circ\epsilon_{B}.$
\begin{defn}
The morphism $S$ is called the \emph{antipode} of the Banach $\cten$-bialgebra
$H,$ if $S\star\id H=\id H\star S=u_{H}\circ\epsilon_{H}$.
\end{defn}
The antipode of a bialgebra is unique, if it exists.
\begin{defn}
A Banach $\cten$-bialgebra $H$ with antipode is called Banach Hopf
$\cten$-algebra. It is a Hopf algebra object in $\Ban$.

Any $\cten$-bialgebra map between two Hopf $\cten$-algebras is also
a Hopf $\cten$-algebra map, i.e. commutes with antipodes. 
\end{defn}

\subsection{\label{sub:Constructions-in-the}Constructions in the category of
Banach $\cten$-bialgebras and Hopf $\cten$-algebras.}

\emph{Banach $\cten$-subbialgebra (Hopf }$\cten$-\emph{subalgebra)}
is a closed subspace, which is a Banach $\cten$-bialgebra (Hopf $\cten$-algebra)
with induced operations.

A closed subspace is a \emph{closed $\cten$-biideal (Hopf }$\cten$-\emph{ideal)
}if it is a $\cten$-ideal and $\cten$-coideal (and is invariant
under the antipode map). A kernel of a Banach $\cten$-bialgebra (Hopf
$\cten$-algebra) morphism is a closed $\cten$-biideal (Hopf $\cten$-ideal).
The quotient by a closed $\cten$-biideal (Hopf $\cten$-ideal) carries
the structure of a Banach $\cten$-bialgebra (Hopf $\cten$-algebra).
Preimage of a closed $\cten$-biideal (Hopf $\cten$-ideal) under
a $\cten$-bialgebra ($ $Hopf $\cten$-algebra) morphism is a closed
$\cten$-biideal (Hopf $\cten$-ideal).

For a Banach $\cten$-bialgebra (Hopf $\cten$-algebra) $H$ its dual
Banach $\cten$-coalgebra $\bhd H$ is a Banach $\cten$-bialgebra
(Hopf $\cten$-algebra) (\cite[Theorem 4]{BD5}).

\section{Banach comodules}

\subsection{Basic definitions}
\begin{defn}
Let $C$ be a Banach $\cten$-coalgebra and $V$ is a Banach space.
We say that $V$ is a right (Banach) $\cten$-comodule over $C$ ($V\in\rbcom C$)
if exists $\rho_{V}:V\to V\hat{\otimes}C$ a $K$-linear continuous
map such that 
\[
\begin{array}{c}
\left(id_{V}\bar{\otimes}\epsilon_{C}\right)\circ\rho_{V}=id_{V}\\
\left(\rho_{V}\otimes id_{C}\right)\circ\rho_{V}=\left(id_{V}\otimes\Delta_{C}\right)\circ\rho_{V}
\end{array}.
\]

\end{defn}
Similarly one defines left Banach $C$-$\cten$-comodules. We denote
the corresponding category by $\lbcom C$.
\begin{defn}
Let $V\in\rbcom C$. Then $\bd V$ is a Banach space which is a right
$\bd C-$$\cten$-module. We call $\bd V$ a \emph{dual }$\cten$-\emph{module}
for $V.$ The action is given by convolution
\[
v'\cdot c'=\left(v'\bar{\otimes}c'\right)\circ\rho_{V}
\]
On $V$ there is also a left $\bd C-$$\cten$-module structure
\[
\begin{array}{ccc}
m:\mbox{ }\bd C\cten V\to V\\
\lambda\otimes v\mapsto\lambda\cdot v & = & \left(id_{V}\bar{\otimes}\lambda\right)\circ\rho_{V}\left(v\right).
\end{array}
\]
We call this $\cten$-module structure \emph{induced}.
\end{defn}
In general, if $W$ is a Banach space, one can give a structure of
a continuous right $\bd C-$$\cten$-module to the space $\Ban\left(V,W\right)$,
similar to the dual $\cten$-module.

If $M,N\in\rbcom C$ and $f\in\Ban\left(M,N\right)$, we say it is
a $\cten$-\emph{comodule morphism}, if $\left(f\cten id\right)\circ\rho_{M}=\rho_{N}\circ f.$
The set of Banach $\cten$-comodule morphisms $\rbcom C\left(M,N\right)$
is a closed $K$-linear subspace of $\Ban\left(M,N\right)$.

\subsection{Constructions in the category of Banach $\cten$-comodules}
\begin{defn}
Let $M\in\rbcom C$. We say that a closed subspace $N\subset M$ is
a \emph{Banach }$\cten$-\emph{subcomodule} (or just closed $\cten$-subcomodule)
if $\rho_{M}\left(N\right)\subset N\cten C$. 

We say that $\cten$-comodule $M$ is \emph{simple}, if it does not
have proper nonzero closed $\cten$-subcomodules.\end{defn}
\begin{prop}
Let $C\in\BCAlg$ and $M\in\rbcom C$. Then
\begin{enumerate}
\item a closed subspace N of M is a right closed $C$-$\cten$-subcomodule
iff N is a closed left $C'$-$\cten$-submodule of M;
\item let $x\in M$. Then the closure of $\bd Cx$ in M is a closed $\cten$-subcomodule
of M and is a space of countable type;
\item if M is simple, then it is of countable type.
\end{enumerate}
\end{prop}
\begin{proof}
\cite[I.1]{bd2}.\end{proof}
\begin{lem}
\cite[II.1]{bd2} Let $f:M\to N$ be a morphism of Banach $\cten$-comodules.
Then 
\begin{enumerate}
\item Preimage of a Banach $\cten$-subcomodule of N is a Banach $\cten$-subcomodule
of M. In particular, $\Ker f$ is a closed $\cten$-subcomodule.
\item The closure of the image of f is a Banach $\cten$-subcomodule of
N.
\end{enumerate}
\end{lem}
\begin{prop}
\cite[2.0.1]{sw} Let f be a morphism of Banach $\cten$-comodules.
\begin{enumerate}
\item If $L\subset M$ is a Banach $\cten$-subcomodule, then $M/L$ is
a Banach $\cten$-comodule and projection is a Banach $\cten$-comodule
morphism;
\item universal property of quotient holds.
\end{enumerate}
\end{prop}
Intersection of two (or many) Banach $\cten$-subcomodules is again
Banach $\cten$-subcomodule.

For a Banach $\cten$-algebra $\left(A,\phi_{A},u_{A}\right)$ we
have defined its dual Banach $\cten$-coalgebra $\left(\bhd A,\bhd{\phi_{A}},\bhd{u_{A}}\right)$,
which is the largest Banach subspace of $\bd A$, which is a Banach
$\cten$-coalgebra w.r.t. $\bd{\phi_{A}}$. One can give a similar
definition for a Banach A -$\cten$-module.

For a right Banach $A$-$\cten$-module $\left(M,\phi_{M}\right)$
with multiplication $\phi_{M}:M\cten A\to M$, let $\bhd M:=\left(\bd{\phi_{M}}\right)^{-1}\left(\bd M\cten\bhd A\right)$.

Similar to the section \ref{sub:Dual-coalgebra}, recall that $\bd M$
is a left $A$-module with left action $\left(a\hits m'\right)\left(m\right)=m'\left(ma\right)$
for $m'\in\bd M$, $m\in M$, $a\in A$. For a fixed $m'$ it defines
a continuous linear map $\hits m':A\to\bd M$, $\hits m'\left(a\right)=a\hits m'$
with $\norm[\hits m']=\norm[m']_{\bd M}$. 
\begin{prop}
Let $m'\in\bhd M$. For $u:M\to N$ a homomorphism of right unital
Banach $\cten$-modules, denote by $\bhd u$ the restriction $u'|_{\bhd N}$.
\begin{enumerate}
\item $m'\in\bhd M$ iff $\hits m':A\to\bd M$ is a completely continuous
\emph{(}i.e. compact\emph{)} operator and $\hits m'\in\bd M\cten\bhd A\subset\bd M\cten\bd A$;
\item $a\hits m'\in\bhd M$ for all $a\in A$, and thus $\hits m':A\to\bhd M$;
\item $\phi_{M}'\left(\bhd M\right)\subset\bhd M\cten\bhd A$;
\item if $u:M\to N$ is a homomorphism of right unital Banach $\cten$-modules,
then $\bhd u\left(\bhd N\right)\subset\bhd M$;
\item for a right Banach $A-\cten$-module $\left(M,\phi_{M}\right)$ the
pair $\left(\bhd M,\bhd{\phi_{M}}\right)$ is a \emph{(}right\emph{)}
Banach $\bhd A-\cten$-comodule.
\end{enumerate}
\end{prop}
\begin{proof}
(1) is similar to \cite[lemma 1]{BD5}

Let $m'\in\bhd M$ and $\phi_{M}'\left(m'\right)=\sum m'_{j}\ten a'_{j}$,
$m_{j}'\in\bd M$, $a_{j}'\in\bhd A$. For all $a\in A$, $m\in M$
we have 
\[
\left(a\hits m'\right)\left(m\right)=\left\langle m',\phi_{M}\left(m\ten a\right)\right\rangle =\left\langle \phi_{M}'\left(m'\right),m\ten a\right\rangle =\sum m'_{j}\left(m\right)\bten a'_{j}\left(a\right).
\]
Thus we have $a\hits m'=\sum m_{j}'\bten a_{j}'\left(a\right)$ and
$\hits m'=\sum m_{j}'\ten a_{j}'\in\bd M\cten\bhd A$.

Conversely, let $\hits m'$ be completely continuous, $\hits m'=\sum m_{j}''\ten a_{j}''\in\bd M\cten\bhd A$.
For all $a\in A$, $m\in M$ 
\[
\left(a\hits m'\right)\left(m\right)=\sum m_{j}''\left(m\right)\ten a_{j}''\left(a\right)=m'\left(ma\right)=m'\left(\phi_{M}\left(m\ten a\right)\right)=\left\langle \phi_{M}'\left(m'\right),m\ten a\right\rangle .
\]
We get $\phi_{M}'\left(m'\right)=\sum m_{j}''\ten a_{j}''\in\bd M\cten\bhd A$
and thus $\phi_{M}'\left(m'\right)\in\bhd M$.

(2) is similar to \cite[lemma 2]{BD5}; One can check that $b\hits\left(a\hits m'\right)=ba\hits m'$.
Let $m'\in\bhd M$ and $\hits m'=\sum m_{j}'\ten a_{j}'\in\bd M\cten\bhd A$.
Then $ba\hits m'=\sum m_{j}'\bten a_{j}'\left(ba\right)=\sum m_{j}'\bten\left(a\hits a_{j}'\right)\left(b\right)$
and thus $\hits\left(a\hits m'\right)=\sum m_{j}'\ten\left(a\hits a_{j}'\right)$.
Since $a_{j}'\in\bhd A$, by proposition \ref{propBHD}.2, $a\hits a_{j}'\in\bhd A$.
Thus $\hits\left(a\hits m'\right)\in\bd M\cten\bhd A$ and by part
1 of this proposition $a\hits m'\in\bhd M$.

(3) follows from (1) and (2); (4) direct check; (5) is similar to
\cite[Theorem 3]{BD5}.\end{proof}
\begin{defn}
We call $\bhd M$ the \emph{dual }$\cten$-\emph{comodule} of $M$.
It is a the largest Banach subspace of $\bd M$, which is a Banach
$\bhd A$-$\cten$-comodule w.r.t. $\bd{m_{M}}$. 

For details, which are similar to the algebraic case, we refer to
\cite{HF}. For a morphism $f:M\to N$ of Banach $\cten$-modules
$\bhd f:\bhd N\to\bhd M$ is a morphism of Banach comodules and thus
we get a functor $\bhd{\left(-\right)}\Bmod A\to\rbcom{\bhd A}$.
If $A=\bd C$, then the functor $\bhd{\left(-\right)}$ is the right
adjoint to the functor $\bd{\left(-\right)}:\rbcom C\to\Bmod{\bd C}$
and preserves finite direct sums. If $M\in\rbcom C$ then $M\subset\bhd{\bd M}$.
\end{defn}
Since $\rbcom C\left(M,N\right)$ is a Banach space, for the pair
of morphisms $f,g\in\rbcom C\left(M,N\right)$ the kernel $\Ker{f-g}$
and the closure of the image $\compl{\Imm{f-g}}$ are equalizer and
coequalizer.

Over a Banach $\cten$-bialgebra $H$ one can define the tensor product
of two right Banach $\cten$-comodules $M$ and $N$ in the usual
way (\cite[I.I.1.1]{bd1}). Namely, it is $\left(M\cten N,\rho_{M\cten N}\right)$
with the coaction 
\[
\rho_{M\cten N}=\left(\Id_{M}\otimes\Id_{N}\otimes m_{H}\right)\circ\left(\Id_{M}\otimes\tau\otimes\Id_{H}\right)\circ\left(\rho_{M}\otimes\rho_{N}\right),
\]
where $\tau\left(a\otimes b\right)=b\otimes a$.

\subsection{Induction}

In the algebraic groups or quantum groups one can describe an induction
functor in terms of cotensor product of comodules over coordinate
algebra (for quantum case see, for example, \cite{L,PW}). Similar
definitions work for Banach $\cten$-coalgebras (bi-, Hopf $\cten$-algebras).

Let $C$ and $B$ be a Banach $\cten$-coalgebras. Let $\left(M,\rho_{M}\right)$
be a right Banach $C-$$\cten$-comodule and let $\left(N,{}_{N}\rho\right)$
be a left Banach $C$-$\cten$-comodule.
\begin{defn}
The space $M\ccten CN=Ker\left(\rho_{M}\otimes id_{N}-id_{M}\otimes{}_{N}\rho\right)$
is called cotensor product of $M$ and $N$ over $C.$ 

Since $M\ccten CN$ is a kernel of a continuous map, it is a closed
subspace of $M\cten N$. Thus if $M$ and $N$ are both Banach spaces,
then $M\ccten CN$ is a Banach space. It is an equalizer of $\rho_{M}\otimes id_{N}$
and $id_{M}\otimes{}_{N}\rho$.
\end{defn}
If $N$ is also a right Banach $B$-$\cten$-comodule, then $M\ccten CN$
is also a right Banach $B$-$\cten$-comodule.

If $f:M\to L$ is a Banach $\cten$-comodule morphism, one has a morphism
$f\ccten CN:M\ccten CN\to L\ccten CN$ of Banach spaces, which is
defined as a restriction $\left(f\cten id_{N}\right)|_{M\ccten CN}$.
Thus cotensoring with a left Banach $\cten$-comodule $N$ gives a
functor $\left(-\ccten CN\right):\rbcom C\to\Ban$. If $N$ is also
a right Banach $B$-$\cten$-comodule, then we have a functor $\left(-\ccten CN\right):\rbcom C\to\rbcom B$.
\begin{defn}
Let $A$ be a Banach $\cten$-algebra. For $M\in\Bmod A$ and $N\in\lbmod A$
define $M\cten_{A}N$ as the quotient Banach space of $M\cten N$
by the closure of the linear hull of the set $\left\{ ma\ten n-m\ten an\right\} $
for all $m\in M,n\in N,a\in A$.
\end{defn}
$M\cten_{A}N$ is an equalizer of the maps $\phi_{M}\ten\id N$ and
$\id M\ten\phi_{N}$, i.e. we have a diagram $M\cten A\cten N\rightrightarrows M\cten N\to M\cten_{A}N\to0$
in $\Ban$. $M\cten_{A}N$ is also a completion of the algebraic tensor
product $M\cten_{A}N$ w.r.t. cross-norm $\norm_{M}\ten\norm_{N}$.
\begin{prop}
\label{prop:ctenemb}For $M\in\rbcom C$, $N\in\lbcom C$ the space
$\bd M\cten_{\bd C}\bd N$ is a closed subspace of $\bd{\left(M\ccten CN\right)}$.\end{prop}
\begin{proof}
Consider the defining exact sequence for $\cten$-cotensor product
\begin{center}
\begin{tikzpicture}[auto]
%%node distance=1.5cm,
\node (o) {$0$};   
\node (c) [right of=o] {$M\ccten C N$};   
\node (cc) [right= and 1cm of c] {$M\cten N$};   
\node (ccc) [right= and 1cm of cc] {$M\cten C\cten N$,};   
\draw[->] (o) to node {} (c);
\draw[->] (c) to node {$\phi$} (cc);
\draw[->] (cc) to node {$\psi$} (ccc);
\end{tikzpicture} 
\end{center}where $\psi=\rho_{M}\cten\id N-\id M\cten\rho_{N}$.

Taking duals gives us the following sequence\begin{center}
\begin{tikzpicture}[auto]    
\node (o) {$0$};   
\node (c) [right= and 0.3cm of o] {$\sd{M\ccten C N}$};   
\node (cc) [right= and 1cm of c] {$\sd{M\cten N}$};   
\node (ccc) [right= and 1cm of cc] {$\sd{M\cten C\cten N}$.};   
\draw[->] (c) to node[swap] {} (o);
\draw[->] (cc) to node[swap] {$\phi'$} (c);
\draw[->] (ccc) to node[swap] {$\psi'$} (cc);
\end{tikzpicture} 
\end{center}The space $\bd{\left(M\cten N\right)}$ contains $\bd M\cten\bd N$as
a closed subspace \cite[4.34]{vanr}. Since $M\ccten CN$ is a complemented
subspace of $M\cten N=\left(M\ccten CN\right)\bigoplus W$, we have
a decomposition of the dual space 
\[
\bd{\left(M\cten N\right)}=\bd{\left(M\ccten CN\right)}\bigoplus\bd W
\]
 and clearly 
\[
\bd W\cong\left(M\ccten CN\right)^{\perp}=\left\{ \phi\in\bd{\left(M\cten N\right)}:\,\phi\left(x\right)=0\;\forall x\in M\ccten CN\right\} .
\]

Since $K$ is discretely valued, all $K$-Banach spaces are free.
Let $M=c_{0}\left(E_{M},K\right)$, $N=c_{0}\left(E_{N},K\right)$,
$C=c_{0}\left(E_{C},K\right)$ and $\left\{ e_{m}\right\} $, $\left\{ e_{n}\right\} $,
$\left\{ e_{c}\right\} $ be the corresponding orthogonal bases. Then
$M\ccten CN=\bigcap\ker(\left(e_{m}\ten e_{c}\ten e_{n}\right)\circ(\rho_{M}\cten\id N-\id M\cten\rho_{N}))$
or, in terms of dual bases, $M\ccten CN=\bigcap\ker\left(e_{m}'e_{c}'\ten e_{n}'-e_{m}'\ten e_{c}'e_{n}'\right)$.
Thus $\left(M\ccten CN\right)^{\perp}$ is equal to the {*}-weak closure
$I^{\wedge*w}$ (\cite[4.7]{R}, which relies on \cite[3.5]{R}, which
can be easily seen via \cite[7.3.1]{PGSch} and \cite[5.1.4, 5.1.6]{PGSch}
or \cite[9.5]{nfa}) of the linear space 
\[
I=\left\{ \sum\left(m'c'\ten n'-m'\ten c'n'\right),\, m'\in\bd M,n'\in\bd N,c'\in\bd C\right\} 
\]
 (actually, just of the elements $\left(e_{m}'e_{c}'\ten e_{n}'-e_{m}'\ten e_{c}'e_{n}'\right)$).
Thus we have a topological isomorphism $\bd{\left(M\cten N\right)}=\bd{\left(M\ccten CN\right)}\bigoplus I^{\wedge*w}$.

Since $\bd M\ten\bd N$ is a free normed space, we can complete an
orthogonal basis $E_{I}$ of $I$ to an orthogonal basis $E_{\bd M\ten\bd N}=E'\cup E_{I}$
of $\bd M\ten\bd N\simeq\left(\bd M\ten_{\bd C}\bd N\right)\bigoplus I$.
Taking norm completions we get $\bd M\cten\bd N\simeq\left(\bd M\cten_{\bd C}\bd N\right)\bigoplus I^{\wedge}$
and taking {*}-weak completion we get $\bd{\left(M\cten N\right)}=\left(M'\cten_{\bd C}N'\right)^{\wedge*w}\bigoplus I^{\wedge*w}$.
Thus we see that inside $\bd{\left(M\cten N\right)}$ we have $\left(\bd M\cten\bd N\right)\cap I^{\wedge*w}=I^{\wedge}$
and that under the quotient map $\bd{\left(M\cten N\right)}/\left(M\ccten CN\right)^{\perp}\to\bd{\left(M\ccten CN\right)}$
we have an embedding $\bd M\cten\bd N/I^{\wedge}=\bd M\cten_{\bd C}\bd N\emb\bd{\left(M\ccten CN\right)}$.
\end{proof}
Let $\phi:C\to B$ be a morphism of Banach $\cten$-coalgebras. Then
$C$ is a left and right Banach $B-$$\cten$-comodule via maps 
\[
\rho_{l\phi}=\left(\phi\otimes\id C\right)\circ\Delta_{C}
\]
 and 
\[
\rho_{r\phi}=\left(\id C\otimes\phi\right)\circ\Delta_{C}.
\]
 Denote those $\cten$-comodules by $_{\phi}C:=\left(C,\rho_{l\phi}\right)$
and $C_{\phi}:=\left(C,\rho_{r\phi}\right)$.
\begin{defn}
Let $M\in\rbcom B$. Then $M\ccten B{}_{\phi}C\in\rbcom C$ is called
the \emph{induced }$\cten$-\emph{comodule.} The other notation is
$M^{\phi}.$ The functor $\left(-\right)^{\phi}$ is called induction.
\end{defn}
\smallskip{}

\begin{defn}
For $M\in\rbcom C$, $M_{\phi}$ will denote the $B-$$\cten$-comodule
$M$ with coaction $M\overset{\rho_{M}}{\to}M\cten C\overset{id\otimes\phi}{\to}M\cten B.$
The functor $\left(-\right)_{\phi}$ is called restriction. \end{defn}
\begin{lem}
\label{lem:-is-a} The map $\id M\bten\epsilon_{C}:M\ccten B{}_{\phi}C\longrightarrow M$
such that $\left(\id M\bten\epsilon_{C}\right)\left(m\otimes c\right)=\epsilon\left(c\right)m$
is a morphism of Banach $B$-$\cten$-comodules $\left(M^{\phi}\right)_{\phi}$
and M. If $C\subset B$ then the image of $\left(\id M\bten\epsilon_{C}\right)\left(M^{\phi}\right)$
is the maximal closed subspace of M, which is an $C$-$\cten$-comodule.\end{lem}
\begin{proof}
the proof is the same as in algebraic case.\end{proof}
\begin{cor}
$M\ccten BB\cong M$ (in fact, $M\ccten BB=\rho_{M}\left(M\right)$).
\end{cor}
For a Banach $\cten$-coalgebra $C$ a finite direct product $C^{n}=\left\{ {\displaystyle \sum_{i=1}^{n}c_{i}e_{i}|\,}c_{i}\in C\right\} $
has a right (and left) Banach $C$-$\cten$-comodule structure with
both right coaction 
\[
\rho_{rC^{n}}\left(\sum_{i=1}^{n}c_{i}e_{i}\right)=\sum\left(c_{i}\right)_{\left(0\right)}e_{i}\ten\left(c_{i}\right)_{\left(1\right)}
\]
 and left coaction 
\[
\rho_{lC^{n}}\left(\sum_{i=1}^{n}c_{i}e_{i}\right)=\sum\left(c_{i}\right)_{\left(0\right)}\ten\left(c_{i}\right)_{\left(1\right)}e_{i}
\]
 defined coordinate-wise.
\begin{defn}
A (right) Banach $C$-$\cten$-comodule $M$ is called \emph{finitely
cogenerated} if it is a closed $\cten$-subcomodule of $C^{n}$.\end{defn}
\begin{prop}
\label{prop:IndFinCogen}Let $M\in\rbcom B$ and $N\in\rbcom C$.
\begin{enumerate}
\item If $M$ is a finitely cogenerated Banach $B$-$\cten$-comodule, then
$M^{\phi}$ is a finitely cogenerated Banach $C$-$\cten$-comodule;
\item If $\phi$ is injective and N is finitely cogenerated, then $N_{\phi}$
is a finitely cogenerated Banach $B$-$\cten$-comodule.
\item if $M$ is finitely generated Banach $\cten$-module over a Banach
$\cten$-algebra $A$, then $\bhd M$ is finitely cogenerated over
$\bhd A$.
\end{enumerate}
\end{prop}
\begin{proof}
We have an embedding $M\hookrightarrow B^{n}$ for some $n\in\mathbb{N}$.
Then $M^{\phi}=M\ccten B{}_{\phi}C\hookrightarrow B^{n}\ccten B{}_{\phi}C=\left(B\ccten B{}_{\phi}C\right)^{n}\cong C^{n}$,
which proves 1). 2) is obvious. 3) is similar to \cite{HF}.\end{proof}
\begin{defn}
A Banach $C$-$\cten$-comodule $M$ is called \emph{cofree} if it
is of the form $M\cong V\cten C$ for a Banach space $V$ with $\cten$-comodule
action $\rho=id_{V}\otimes\cm_{C}$.
\end{defn}
Cofree Banach $\cten$-comodules are cofree objects in the category
of right Banach $C$-$\cten$-comodules over the Banach space $V$
with the covering map $p:V\cten C\to V$ being $p=id_{V}\bar{\otimes}\epsilon_{C}$.
If $V$ is finite-dimensional, we have a $\cten$-comodule isomorphism
$V\cten C\cong C^{\dim_{K}V}$.

If $M$ is free finitely generated Banach $\cten$-module over a Banach
$\cten$-algebra $A$, then $\bhd M$ is cofree finitely cogenerated
over $\bhd D$ (similar to \cite{HF}).

As in the algebraic case, our induction and restriction functors are
related by Frobenius reciprocity.
\begin{prop}
\label{pro:(Frobenius-reciprocity)}(Frobenius reciprocity)

Let $\pi\in\BCAlg\left(C,B\right)$, be a morphism of K-Banach $\cten$-coalgebras,
$M\in\rbcom B$ and $N\in\rbcom C$.

There is a topological isomorphism 
\[
\rbcom C\left(N,M\ccten B{}_{\pi}C\right)\cong\rbcom B\left(N_{\pi},M\right).
\]
\end{prop}
\begin{proof}
The morphisms of $\cten$-comodules 
\[
\begin{array}{ccccc}
 & \phi & \longmapsto & \left(\id M\bar{\otimes}\epsilon_{C}\right)\circ\phi & =\tilde{\phi}=\left(\id M\bar{\otimes}\left(\epsilon_{B}\circ\pi\right)\right)\circ\phi\\
\left(\psi\otimes\id C\right)\circ\rho_{N}= & \bar{\psi} & \longleftarrow & \left(\psi:N\to M\right)
\end{array}.
\]
are inverse to each other \cite[Prop.6]{DOI}.

Since the topology on spaces $\rbcom C\left(N,M\ccten B{}_{\pi}C\right)$
and $\rbcom B\left(N_{\pi},M\right)$ is induced from $L_{b}\left(N,M\cten C\right)$
and $L_{b}\left(N,M\right),$ the continuity of our linear bijections
follows from the argument same as in \cite[sec.18]{nfa}. Namely,
composition with linear continuous map $W\to U$ is a linear continuous
map $L_{b}\left(V,W\right)\to L_{b}\left(V,U\right).$ Since our $\left(\tilde{\cdot}\right)$
maps are compositions of continuous maps, they are continuous.
\end{proof}

\subsection{\label{sub:Rational-modules}Rational $\cten$-modules}

Let $C\in\BCAlg$ be a Banach $\cten$-coalgebra and $M$ be a left
Banach $\cten$-module over $C'.$ Every element $m\in M$ defines
a map from $C'$ to $M$ through the cyclic $\cten$-module, generated
by $m.$ Following \cite[2.1]{sw}, taken altogether for all $m\in M$,
this defines a continuous map 
\begin{equation}
\begin{array}{ccccc}
\rho: & M & \to & \Ban\left(C',M\right)\\
 & m & \mapsto & \rho\left(m\right): & \rho\left(m\right)\left(c'\right)=c'm
\end{array}.\label{eq:ratmod}
\end{equation}
There is a natural embedding
\[
\begin{array}{cccc}
M\cten C & \hookrightarrow & \Ban\left(C',M\right)\\
m\otimes c & \longmapsto & f\left(m\otimes c\right): & f\left(m\otimes c\right)\left(c'\right)=c'\left(c\right)\cdot m
\end{array}.
\]

Let $M\in\rbcom C$ and consider the induced $C'-$$\cten$-module
structure on $M$. In this case $c'\cdot m=\left(\id M\bten c'\right)\circ\rho_{M}\left(m\right).$
Then our map $\rho:M\to\Ban\left(C',M\right)$ acts on $m$ exactly
as $\rho_{M}$ (since the results coincide on every element of $C'$),
and thus $\rho\left(m\right)\in M\cten C.$
\begin{defn}
In the notations above
\begin{itemize}
\item $M$ is called t-rational if $\rho\left(M\right)\subset M\cten C$;
\item $M$ is called rational if $\rho\left(M\right)\subset M\otimes C$.
\end{itemize}
\end{defn}
In both cases $\rho\left(M\right)$ satisfy axioms of a $\cten$-comodule
action.

In other words, a $\cten$-module is t-rational if its $\cten$-module
structure is induced from a Banach $\cten$-comodule structure and
rational if it is induced from a $\ten$-comodule.
\begin{rem}
Since for any Banach spaces $M$ and $C$ by \cite[18.11]{nfa} we
have the inclusion $M\cten C\subset M\cten C''\cong C\left(C',M\right)$.
Thus if $M$ is t-rational (rational) then $\rho\left(M\right)\subset C\left(C',M\right),$i.e.
$\rho\left(m\right)$ is a compact (finite rank) map for every $m\in M$.\end{rem}
\begin{prop}
\cite[2.1.3]{sw} Let L, M, N $\in C'-mod$ and M, N being t-rational
(rational).
\begin{enumerate}
\item if $N\subset M$ is a $\cten$-subcomodule \emph{(}$\ten$-subcomodule\emph{)},
then N is also a $\cten$-submodule;
\item every cyclic $\cten$-submodule is countable type \emph{(}finite dimensional\emph{);}
\item any quotient of a t-rational \emph{(}rational\emph{)} module is t-rational
\emph{(}rational\emph{);}
\item L has a unique maximal rational submodule 
\[
L^{rat}=\sum\mbox{"all rational submodules"}=\rho_{L}^{-1}\left(L\otimes C\right)
\]
 and a unique maximal t-rational $\cten$-submodule 
\[
L^{t-rat}=\widehat{\sum}\mbox{"all t-rational \ensuremath{\cten}-submodules"}=\rho_{L}^{-1}\left(L\cten C\right);
\]

\item any homomorphism of t-rational \emph{(}rational\emph{)} modules is
a $\cten$-comodule \emph{(}$\ten$-comodule\emph{)} morphism.
\end{enumerate}
\end{prop}
\begin{cor}
Any finitely generated rational module is finite dimensional.
\end{cor}

\subsection{Tensor identities}
\begin{prop}
\label{pro:(Tensor-identities}\emph{(}Tensor identities\emph{)} Let
$\pi\in\mbox{BHopf}_{K}\left(H,B\right)$, $W\in\rbcom B$ and $V\in\rbcom H$.
Then we have the following isomorphisms of Banach $H$-$\cten$-comodules.:
\begin{enumerate}
\item $V\cten\left(W\ccten B{}_{\pi}H\right)\cong\left(V_{\pi}\cten W\right)\ccten B{}_{\pi}H$ 
\item $\left(W\ccten B{}_{\pi}H\right)\cten V\cong\left(W\cten V_{\pi}\right)\ccten B{}_{\pi}H$
\end{enumerate}
\end{prop}
\begin{proof}
The maps 
\[
\begin{array}{cccc}
\phi: & V\cten\left(W\ccten B{}_{\pi}H\right) & \to & \left(V_{\pi}\cten W\right)\ccten B{}_{\pi}H\\
 & \sum v\otimes w\otimes h & \longmapsto & \sum v_{(0)}\otimes w\otimes v_{(1)}h
\end{array}.
\]
and 
\[
\begin{array}{cccc}
\psi: & \left(V_{\pi}\cten W\right)\ccten B{}_{\pi}H & \longrightarrow & V\cten\left(W\ccten B{}_{\pi}H\right)\\
 & \sum v\otimes w\otimes h & \longmapsto & \sum v_{(0)}\otimes w\otimes S_{H}\left(v_{(1)}\right)h
\end{array}.
\]
are morphisms of right Banach $H$-$\cten$-comodules that are inverse
to each other\cite{DOI}. Since all involved maps are continuous,
$\phi$ and $\psi$ are continuous.

The proof of (2) is similar.\end{proof}
\begin{cor}
\label{cor:triv}If $B=K,$ $\pi=\epsilon_{H}$ and $W=K$ then we
have a $\cten$-comodule isomorphism 
\[
V\cten H\cong V_{tr}\cten H,
\]
where $V_{tr}$ means the underlying vector space of V with trivial
$H-\cten$-comodule structure \emph{(}$\rho_{V_{tr}}\left(v\right)=v\ten1$\emph{)}.
\end{cor}

\section{Locally convex $\cten$-coalgebras}

\subsection*{Preliminaries}

\subsubsection*{Basic notions}

For the background on nonarchimedean functional analysis we refer
mostly to \cite{nfa} and \cite{PGSch}.

Let \emph{V} be Topological Vector Space (TVS) over $K$.

A subset of $V$ is called \emph{absolutely convex} if it is an $K^{0}$-submodule
of $V$.

A subset of $V$ is called \emph{convex} if it is of the form $x+A$,
$x\in V$ and $A$ is an absolutely convex set.

An absolutely convex subset $L\subset V$ is called a \emph{absorbing}
if $\forall x\in V$ $\exists\lambda\in K^{*}$ s.t. $x\in\lambda L$.

A \emph{lattice} is an absolutely convex absorbing set.

A TVS is Locally Convex (LCTVS) if its topology is generated by lattices.
We use LCTVS both as abbreviation and as the notation for the corresponding
category.

The strong dual of an LCTVS $V$ is denoted by $\sd V$.

\subsubsection*{Inductive and projective limits}

Let $\mathcal{F}=\left\{ f_{h}:V_{h}\to V\right\} _{h\in H}$ be a
family of maps of LCTVSs. A \emph{final locally convex topology} $\tau_{fin}\left(\mathcal{F}\right)$
on $V$ w.r.t. $\mathcal{F}$ is the strongest locally convex topology
that makes all $f_{h}\in\mathcal{F}$ continuous. Such a topology
always exists.

Let $\left(V_{n},\phi_{nm}\right)_{n\in\mathbb{N}}$ be an inductive
sequence of LCTVSs. A locally convex inductive limit of $\left(V_{n},\phi_{nm}\right)_{n\in\mathbb{N}}$
is its algebraic inductive limit $V=\ilim V_{n}$, equipped with final
locally convex topology w.r.t. $\phi_{n}:V_{n}\to V$. Any family
of morphisms $f_{n}:V_{n}\to U$ on $\left(V_{n},\phi_{nm}\right)_{n\in\mathbb{N}}$
in LCTVS, compatible with $\phi_{nm}$, factors through $V$.

$V$ can also be described as a quotient of the locally convex direct
sum $\bigoplus V_{n}$ by the subspace $D$, generated by vectors
$\left\{ \left(\dots,v_{n},-\phi_{n,n+1}\left(v_{n}\right),\dots\right)|v_{n}\in V_{n}\right\} $.
We denote by $\pi:\bigoplus V_{n}\to V\simeq\bigoplus V_{n}/D$ the
quotient map. The universal property of $V$ follows from the universal
property of $\bigoplus V_{n}$ .

A map $f:V\to U$ in LCTVS is continuous iff maps $f\circ\phi_{n}$
are continuous. Taking $f=\id V$ one get that subset $F\subset V$
is open (closed) if and only if $\phi_{n}^{-1}\left(F\right)$ is
open (closed) \cite[1.1.6]{GPKS} for all $n$.

Similarly one can define locally convex projective limits. Let $\mathcal{F}=\left\{ f_{h}:V\to V_{h}\right\} _{h\in H}$
be a family of maps of LCTVSs. An \emph{initial locally convex topology}
$\tau_{in}\left(\mathcal{F}\right)$ on $V$ w.r.t. $\mathcal{F}$
is the weakest locally convex topology that makes all $f_{h}\in\mathcal{F}$
continuous. 

Let $\left(V_{n},\phi_{nm}\right)_{n\in\mathbb{N}}$ be a projective
sequence of LCTVSs. A locally convex projective limit of $\left(V_{n},\phi_{nm}\right)_{n\in\mathbb{N}}$
is its algebraic projective limit $V=\plim V_{n}$, equipped with
final locally convex topology w.r.t. $\phi_{n}:V\to V_{n}$. $V$
can be described as a subspace of the locally convex direct product
$\prod V_{n}$, consisting of vectors $\left\{ \left(\dots,v_{n},v_{n+1},\dots\right),v_{n}\in V_{n}|\phi_{n+1,n}\left(v_{n+1}\right)=v_{n}\right\} $.

\smallskip{}

\subsubsection*{Limits of sequences of spaces}
\begin{defn}
Let \emph{V} be an LCTVS.
\begin{itemize}
\item \emph{V} is a \emph{LB-space}, if it is an locally convex inductive
limit of a sequence $\left(V_{n},\phi_{nm}\right)$ of Banach spaces.
\item \emph{V} is a \emph{Fréchet} if it is a locally convex projective
limit of a sequence $\left(V_{n},\phi_{nm}\right)$ of Banach spaces.
\item V is an \emph{LS-space} if it is an LB-space and the transition maps
$\phi_{n}$ are compact. 
\item V is an \emph{FS (Fréchet-Schwartz) space}, if it is a Fréchet space
and the transition maps $\phi_{n}$ are compact. 
\end{itemize}
\end{defn}
In another terminology, LS-spaces are called compact type (CT) spaces,
and FS-spaces are called nuclear Fréchet (NF) spaces.
\begin{prop}
Any LS-space $V$ is Hausdorff, complete and reflexive \cite[Theorem 1.1]{ST4}.

Every Fréchet space is Hausdorff and complete \cite[Proposition 8.2]{nfa}.
FS-spaces are also reflexive.

A closed subspace of an LS (FS) space is an LS (FS) space. The quotient
of an LS (FS) space by a closed subspace is an LS (FS) space \cite[1.2]{ST4}.

The strong dual of an LS-space is an FS-space and vice verse \cite[1.3]{ST4}.
Thus the categories of LS and FS spaces are antiequivalent.
\end{prop}
The following version of the Open Mapping theorem follows from the
general theorem of De Wilde \cite[24.30]{MEI}.
\begin{thm}
\label{thm:(Open-Mapping-theorem)}(Open Mapping theorem) Let $f:E\to U$
be a continuous surjective map of LCTVS $E$ and $U$. If $E$ is
either LB- or Fréchet space, or their closed subspace or quotient,
and if U is an LB- or Fréchet space (in particular, Banach), then
$f$ is an open map.
\end{thm}

\subsubsection*{Tensor products.}

The category of LCTVS has several natural topologies on the tensor
product of two LCTVS, most commonly used are the inductive $\otimes_{i,K}$,
injective $\otimes_{\epsilon,K}$ and projective $\otimes_{\pi,K}$
tensor product topologies. Over discretely valued fields, injective
and projective tensor products coincide for all LCTVS \cite[10.2]{PGSch}.
For Fréchet or sequentially complete LB spaces (and possibly in more
general cases too) inductive and projective tensor products coincide
\cite[17.6]{nfa},\cite[1.1.31]{EM}. Thus in our cases of interest
(LS and FS spaces) we have only one reasonable tensor product topology,
which after completion give rise to the completed tensor product $\cten$.

The categories of LS and FS spaces are tensor categories with respect
to the $\cten$. 
\begin{prop}
\label{prop:topten}Let $V\cong\ilim V_{n}$, $U\cong\ilim U_{n}$
be LS-spaces, $F=\plim F_{n}$ , $H=\plim H_{n}$ be FS-spaces and
$W$ be a Banach space.
\begin{enumerate}
\item $V\cten U=\ilim V_{n}\cten U_{n}$ \cite[1.1.32]{EM};
\item $F\cten H=\plim F_{n}\cten H_{n}$ \cite[1.1.29]{EM};
\item $\sd{V\cten U}=\sd V\cten\sd U$, $\sd{F\cten H}=\sd F\cten\sd H$
\cite[20.13, 20.14]{nfa};
\item \label{enu:int}$L_{b}\left(U,V\right)\cong\sd U\underset{K,\pi}{\cten}V$
and $L_{b}\left(\sd U,V\right)\cong U\cten V$ \cite[20.9, 20.12]{nfa};
\item $V\otimes_{K,i}W\cong V\otimes_{K,\pi}W$ and thus $V\cten_{K,i}W\cong V\cten_{K,\pi}W$
\cite[1.1.31]{EM};
\item \label{enu:closedLB}We have a topological isomorphism $V\cten W\cong\ilim\left(V_{n}\cten W\right)$\textup{
(see appendix)};
\item \textup{\label{enu:-and-CTBten}$\sd{V\cten W}=\sd V\cten\sd W$ and
$\sd{F\cten W}=\sd F\cten\sd W$ (see appendix).}
\end{enumerate}
\end{prop}

\subsection{Topological Coalgebras}
\begin{defn}
Let's first review related basic definitions of the theory of topological
algebras (and set the terminology, which can differ in the literature).
\begin{itemize}
\item An LCTVS algebra $A$ is an LCTVS with a continuous multiplication
map.
\item An F-algebra $A$ is a topological algebra, which is a Fréchet space.
\item An FS-algebra $A$ is an F-algebra, which is a nuclear Fréchet space.
\item A Fréchet algebra $A$ is an F-algebra, whose topology can be given
by a (countable) system of submultiplicative seminorms. In this case
it is a locally convex projective limit of Banach $\cten$-algebras
with transition maps being $\cten$-algebra morphisms.
\item A Nuclear Fréchet (NF) algebra $A$ is a Fréchet algebra, which is
a nuclear space.
\end{itemize}
\end{defn}
One can require the multiplication on LCTVS algebra only to be separately
continuous. In this case, it is a monoid in the tensor category (LCTVS,
$\otimes_{i,K}$) with the tensor structure given by inductive tensor
product. In our definition, it is a monoid w.r.t projective tensor
product $\otimes_{\pi,K}$. For Fréchet spaces these two notions coincide,
since any separately continuous map in this case is (jointly) continuous
and an F-algebra is an algebra in the tensor category of Fréchet spaces
with the tensor structure given by $\cten$.

Every Fréchet space $A$ can be presented as a locally convex projective
limit of Banach spaces , i.e. ${\displaystyle A=\plim A_{n}}$. If
$A$ is a Fréchet algebra, $A_{n}$ can be chosen to be Banach algebras
(Arens-Michael presentation) and if $A$ is an NF-algebra, one can
chose $\left(A_{n},\phi_{nm}\right)$ such that transition maps $\phi_{nm}$
are compact. In general, for multiplication to be continuous only
a family form of submultiplicativity is required, i.e. $\norm[x\cdot y]_{n}\leq\norm[x]_{n+k}\norm[y]_{x+k}$
(see \cite{G}) and thus for an F-algebra an Arens-Michael presentation
might not exist.

Our main object of interest is the category, opposite to the category
of NF-algebras. For an NF-algebra we will call a corresponding Arens-Michael
system an \emph{NF-structure}. It is known that any two NF-structures
are equivalent \cite[1.2.7]{EM} in the category of projective systems
of Banach spaces.
\begin{defn}
Let $C$ be an LCTVS. 
\begin{itemize}
\item For a tensor structure $\cten$ ($\cten_{i,K}$, $\cten_{\pi,K}$)
on LCTVS, $C$ is an \emph{LCTVS }$\cten$\emph{-coalgebra} if it
is a $\cten$-coalgebra object in the tensor category (LCTVS,$\cten$). 
\item $C$ is LS(LB)-$\cten$-coalgebra if it is an LCTVS $\cten$-coalgebra
and an LS(LB)-space.
\item $C$ is a \emph{CT (Compact Type) }$\cten$-\emph{coalgebra}, if it
is an LS-coalgebra, that is topologically isomorphic to a compact
locally convex inductive limit of a compact inductive system $\left(C_{n},\phi_{nm}\right)$
of Banach $\cten$-coalgebras with transition maps $\phi_{nm}$ being
$\cten$-coalgebra morphisms. We say that $\left(C_{n},\phi_{nm}\right)$
gives $C$ a \emph{CT-structure}.
\end{itemize}
\end{defn}
If $\left\{ C_{n},\phi_{n}\right\} $ be an inductive system of K-Banach
$\cten$-coalgebras $C_{n}$ with injective transition maps $\phi_{n}:C_{n}\to C_{n+1},$
s.t. $\phi_{n}$ are morphisms of $K-$Banach $\cten$-coalgebras,
then $C={\displaystyle \ilim C_{n}}$ is a $\cten$-coalgebra in the
category of locally convex K-vector spaces, with $\cten$-coalgebra
maps $(\Delta,\epsilon)$ defined by the corresponding maps $(\Delta_{n},\epsilon_{n})$.
Thus $C_{n}$ indeed define an LB-$\cten$-coalgebra structure in
$C$. We remind again that for sequentially complete (equivalently,
regular) LB-spaces $\cten_{i,K}$ and $\cten_{\pi,K}$ coincide. In
particular, it is true for LS-spaces.

Since for any inductive system one can construct the one with the
same inductive limit and injective transition maps, without loss of
generality we can only consider CT-structures with injective transition
maps.
\begin{defn}
We say that two CT-structures $C=\ilim C_{n}$ and $C=\ilim D_{n}$
are equivalent if they are isomorphic in the category of inductive
systems of Banach $\cten$-coalgebras.\end{defn}
\begin{lem}
For a CT-$\cten$-coalgebra $C$ any two CT-structures with injective
transition maps are equivalent.\end{lem}
\begin{proof}
Let $\left(C_{n},\phi_{nm}\right)$ and $\left(D_{n},\psi_{nm}\right)$
be two CT-structures for $C$. Since $C=\ilim C_{n}$ and $C=\ilim D_{n}$,
$\left(C_{n},\phi_{nm}\right)$ and $\left(D_{n},\psi_{nm}\right)$
are equivalent as inductive systems of Banach spaces. Since $\phi_{nm}$
and $\psi_{nm}$ are injective, the embeddings $\phi_{n}:C_{n}\to C$
and $\psi_{n}:D_{n}\to D$ are also injective and therefore the above
equivalence maps must be Banach $\cten$-coalgebra morphisms.\end{proof}
\begin{rem}
\label{LSdualFS}Every LS-$\cten$-coalgebra is a $\cten$-coalgebra
object in the category of LS-spaces. Since the dual of a completed
tensor product $C\cten D$ of two LS-spaces $C$ and $D$ is the completed
tensor product $C'_{b}\cten D'_{b}$ of their strong duals, the duality
functor maps the commutative diagrams, defining $\cten$-coalgebra
structure for $C$, into the diagrams, which satisfy $\cten$-algebra
axioms for $C_{b}'$. Thus the dual of a $\cten$-coalgebra object
in the category of LS-spaces is an $\cten$-algebra object in the
category of FS-spaces and we have an antiequivalence of categories
of LS-$\cten$-coalgebras and FS-$\cten$-algebras.
\end{rem}
Clear that if $C$ is a CT-$\cten$-coalgebra $C=\lim_{\to}C_{n}$,
then $C_{b}'$ is an NF-$\cten$-algebra with NF-structure $\left(C'_{n},\phi'_{nm}\right)$.
\begin{prop}
\label{nfdualct}Let $A=\left(A,m,u\right)$ be an NF-$\cten$-algebra
with NF-structure $\left(A_{n},\phi_{nm}\right)$, $A\cong\plim A_{n}$.
Then $A'_{b}$ is a CT-$\cten$-coalgebra $\left(A'_{b},\bd m,\bd u\right)$
with a CT-structure $\left(\bhd{A_{n}},\bhd{\phi_{nm}}\right)$.\end{prop}
\begin{proof}
The only thing one needs to prove is that $\ilim\bhd{A_{n}}=\sd A$.
Recall that $\bhd{A_{n}}=\left(\bd{m_{A_{n}}}\right)^{-1}\left(A_{n}'\cten A_{n}'\right)$.
Since $\bhd{\phi_{nm}}$ are Banach $\cten$-coalgebra morphisms,
$\bhd{A_{n}}$ form a compact inductive system of Banach $\cten$-coalgebras,
with the locally convex inductive limit $\ilim\bhd{A_{n}}\subset\sd A$
being the preimage $\left(\bd m\right)^{-1}\left(\sd A\cten\sd A\right)$.
But since $\left(\sd A\cten\sd A\right)=\sd{A\cten A}$, the preimage
of the coaction is the whole $\sd A$, which is itself an LS-$\cten$-coalgebra.
\end{proof}
Since spaces of compact type are reflexive, the above proposition
gives us an antiequivalence of categories
\[
\begin{array}{ccc}
\left\{ \mbox{CT-\ensuremath{\cten}-coalgebras}\right\}  & \longleftrightarrow & \left\{ \mbox{NF-\ensuremath{\cten}-algebras}\right\} \end{array}
\]
with morphisms being continuous topological $\cten$-coalgebra and
$\cten$-algebra morphisms.

The notions of $\cten$-subcoalgebra, $\cten$-left coideal, $\cten$-right
coideal and $\cten$-coideal (2-sided) for topological $\cten$-coalgebras
are defined similarly to the Banach case.
\begin{prop}
Let $V$ be a CT-$\cten$-coalgebra.
\begin{enumerate}
\item If $U$ is a closed $\cten$-subcoalgebra, then $U$ is also of compact
type. 
\item If $I$ is a closed $\cten$-coideal of $V,$ then $V/I$ is a CT-$\cten$-coalgebra.
\item If $f:V\to W$ is a morphism of (topological) CT-$\cten$-coalgebras
then $\Ker f$ is closed $\cten$-coideal.
\item If J is a closed $\cten$-coideal of W then $f^{-1}\left(J\right)$
is a closed $\cten$-coideal.
\item The closure of $f\left(V\right)$ is a closed CT-$\cten$-subcoalgebra.
\end{enumerate}
\end{prop}
\begin{proof}
$U\subset V$ is closed iff $U_{n}=U\cap V$ is closed $\forall n$.
Thus $U_{n}$ are Banach subspaces of $V_{n}$ and, since $U$ is
a $\cten$-subcoalgebra, are K-Banach $\cten$-subcoalgebras of $V_{n}.$
So $U={\displaystyle \ilim U_{n}}$ is a CT-$\cten$-coalgebra.

Same argument works for $V/I={\displaystyle \ilim V_{n}/I_{n}}$ with
$I_{n}=I\cap V_{n}$, which proves (2).

The proofs of the rest statements are the same as in Banach case.
\end{proof}

\subsection{Topological Bialgebras and Hopf algebras.}

We have similar definitions of an LCTVS-, LS-, FS-, CT- and NF-$\cten$-bialgebras.

Similar to the remark \ref{LSdualFS}, the categories of LS- and FS-$\cten$-bialgebras
are antiequivalent under the duality map $V\mapsto\sd V$. To see
if the categories of CT-$\cten$-bialgebras and NF-$\cten$-bialgebras
are equivalent, we will first establish an auxiliary result. 
\begin{lem}
\label{lem:subset}Let F and H be FS-spaces with presentations $F=\plim F_{n}$
and $H=\plim H_{n}$. Let $f:F\to H$ be a morphism, which is defined
by a morphism of projective systems $\left\{ f_{n}:F_{n}\to H_{n}\right\} $.
Let $U=\plim U_{n}$, $U_{n}\subset H_{n}$, be a closed subspace
of $H$. Then $f^{-1}\left(U\right)=\plim f_{n}^{-1}\left(U_{n}\right)$.\end{lem}
\begin{proof}
Follows from definitions unpacking.\end{proof}
\begin{prop}
\label{prop:DualNFisCT}For a CT-$\cten$-bialgebra $B\cong\ilim B_{n}$
with CT-structure $\left(B_{n},\rho_{nm}\right)$, the dual system
$\left(\bhd{B_{n}},\bhd{\rho_{nm}}\right)$ forms an NF-structure
for $\sd B$. Thus $\sd B$ is an NF-$\cten$-bialgebra.\end{prop}
\begin{proof}
Clear that $\left(\bhd{B_{n}},\bhd{\rho_{nm}}\right)$ is an NF-system
of Banach $\cten$-bialgebras. The proof that its projective limit
is $\sd B$ is similar to the Proposition \ref{nfdualct}. The maps
$\bd{m_{B_{n}}}:\bd{B_{n}}\to\bd{\left(B_{n}\cten B_{n}\right)}$
define a morphism of projective systems $\left\{ \bd{B_{n}}\right\} $
and $\left\{ \bd{\left(B_{n}\cten B_{n}\right)}\right\} $. The tensor
products $\bd{B_{n}}\cten\bd{B_{n}}\subset\bd{\left(B_{n}\cten B_{n}\right)}$
form a projective subsystem and thus define a closed subspace of $\sd{B\cten B}$.
By preceding lemma \ref{lem:subset}, $\left(\bd{m_{B}}\right)^{-1}\sd{B\cten B}=\left(\bd{m_{B}}\right)^{-1}\left(\plim\bd{B_{n}}\cten\bd{B_{n}}\right)$
is equal to the $\plim\left(\bd{m_{B_{n}}}\right)^{-1}\left(\bd{B_{n}}\cten\bd{B_{n}}\right)$.
Since, by nuclearity, $\sd{B\cten B}=\sd B\cten\sd B$, we have $\sd B=\left(\bd{m_{B}}\right)^{-1}\left(\sd B\cten\sd B\right)$.
All together it gives us 
\[
\sd B=\left(\bd{m_{B}}\right)^{-1}\left(\sd{B\cten B}\right)=\left(\bd{m_{B}}\right)^{-1}\left(\sd B\cten\sd B\right)=\left(\bd{m_{B}}\right)^{-1}\left(\plim\bd{B_{n}}\cten\bd{B_{n}}\right)=
\]
\[
=\plim\left(\bd{m_{B_{n}}}\right)^{-1}\left(\bd{B_{n}}\cten\bd{B_{n}}\right)=\plim\bhd{B_{n}}.
\]
Thus $\left(\bhd{B_{n}},\bhd{\rho_{nm}}\right)$ defines an NF-structure
on $\bd B$.
\end{proof}
Thus we have an equivalence of categories

\[
\begin{array}{ccc}
\left\{ \mbox{CT-\ensuremath{\cten}-bialgebras}\right\}  & \longleftrightarrow & \left\{ \mbox{NF-\ensuremath{\cten}-bialgebras}\right\} \end{array}.
\]

LCTVS-, LS-, FS-, CT- and NF-Hopf $\cten$-algebras are defined similarly.
One can easily check that, similar to algebraic and Banach case, the
antipode on a topological bialgebra is unique if exists, and that
the proposition \ref{prop:DualNFisCT} still holds. Thus we also have
an equivalence of categories

\[
\begin{array}{ccc}
\left\{ \mbox{CT-Hopf \ensuremath{\cten}-algebras}\right\}  & \longleftrightarrow & \left\{ \mbox{NF-Hopf \ensuremath{\cten}-algebras}\right\} \end{array}.
\]

The notions of $\cten$-subbialgebra (Hopf $\cten$-subalgebra), left
$\cten$-biideal (Hopf $\cten$-ideal), right $\cten$-biideal (Hopf
$\cten$-ideal) and $\cten$-biideal (Hopf $\cten$-ideal) (2-sided)
can be defined similar to the Banach case.

The quotient of a (CT-) $\cten$-bialgebra (Hopf $\cten$-algebra)
by a closed $\cten$-biideal (Hopf $\cten$-ideal) is a (CT-) $\cten$-bialgebra
(Hopf $\cten$-algebra).

The kernel of the morphism of $\cten$-bialgebras (Hopf $\cten$-algebras)
is closed $\cten$-biideal (Hopf $\cten$-ideal).

Preimage of a closed $\cten$-biideal (Hopf $\cten$-ideal) is a closed
$\cten$-biideal (Hopf $\cten$-ideal).

Closure of the image of a morphism of $\cten$-bialgebras (Hopf $\cten$-algebras)
is a closed $\cten$-subbialgebra (Hopf $\cten$-subalgebra). The
proof is the same as in Banach case.
\begin{example}
For any open compact subgroup $G$ of a locally analytic $K$-group
$\mathbb{G}$ the algebra of locally analytic $K$-valued functions
$C^{la}\left(G,K\right)$is a commutative CT-Hopf $\cten$- algebra.
Its is known that $G$ possesses a system of neighborhoods of zero
$\left\{ H_{i}\right\} $ consisting of open normal compact subgroups
$H_{i}$. The CT-structure is given by Banach Hopf $\cten$-algebras
$C_{H_{i}}^{la}\left(G,K\right)$ of locally $H_{i}$-analytic functions,
i.e. the functions $f\in C^{la}\left(G,K\right)$ s.t. $f|_{gH_{i}}$
is (rigid) analytic for all $g\in G$. The strong dual space $D^{la}\left(G,K\right)=\sd{C^{la}\left(G,K\right)}$
is a cocommutative NF-Hopf $\cten$-algebra.

For a worked out noncommutative noncocommutative example we refer
to \cite{Lyu}.
\end{example}

\section{modules and comodules}

\subsection{Definitions}
\begin{defn}
Let $C$ be a CT-coalgebra and $V$ is an LCTVS. Then we say that
\begin{enumerate}
\item $V$ is a right \emph{LCTVS }$\cten_{\pi}$-\emph{comodule} over $C$
($V\in\mbox{LCVSComod}_{C}$) if there exists a $K$-linear continuous
map $\rho_{V}:V\to V\cten_{\pi}C$ such that 
\[
\begin{array}{c}
\left(id_{V}\bar{\otimes}\epsilon_{C}\right)\circ\rho_{V}=id_{V}\\
\left(\rho_{V}\otimes id_{C}\right)\circ\rho_{V}=\left(id_{V}\otimes\Delta_{C}\right)\circ\rho_{V}
\end{array}.
\]

\item $V$ is an\emph{ LS(LB)-}$\cten$-\emph{comodule} over $C$ if $V$
is an LCTVS comodule and an LS(LB)-space.
\item $V$ is a \emph{CT-}$\cten$-\emph{comodule} over $C$ if $V$ is
an LS-comodule over $C$ and is isomorphic to the locally convex inductive
limit $V\cong\ilim V_{n}$ of a compact inductive system $\left(V_{n},\psi_{n,m}\right)$
of Banach comodules $\left(V_{n},\rho_{V_{n}}\right)$ over $C_{n}$,
where $\left(C_{n},\phi_{nm}\right)$ is a CT-structure for C. For
a fixed $\left(C_{n},\phi_{nm}\right)$ we will also say that $V$
is a $\left(C_{n},\phi_{nm}\right)$-comodule.
\end{enumerate}
\end{defn}
As morphisms of LCTVS-comodules we take continuous comodule maps.
\begin{defn}
Let $A$ be an NF-$\cten$-algebra and $V$ be an LCTVS. Then we say
that
\begin{enumerate}
\item $V$ is a right \emph{LCTVS-}$\cten$-\emph{module} over $A$ if it
is a $\cten$-module over $A$.
\item $V$ is a \emph{right FS-}$\cten$-\emph{module} over $A$, if, in
addition, $V$ is a FS-space.
\item $V$ is an \emph{NF-}$\cten$-\emph{module} over $A$, if it is a
projective limit of a compact projective system $\left(V_{n}\right)_{n\in\mathbb{N}}$
of Banach $A_{n}$-$\cten$-modules $V_{n}$, where $\left(A_{n}\right)$
is an NF-structure for $A$.
\end{enumerate}
\end{defn}
By duality, for an LCTVS-$\cten$-comodule $V$, $\bd V=\sd V$ is
an LCTVS which is a right $\sd C-$module. 
\begin{defn}
We will say that $V'$ is a $\cten$-module, dual to $\cten$-comodule
$V.$ 
\end{defn}
Note that if $C$ is a CT-Hopf $\cten$-algebra, then $\sd C$ is
an NF-Hopf $\cten$-algebra.

Clearly the dual of an LS-$\cten$-comodule is a FS-$\cten$-module
and vice versa. The duality map gives an anti-equivalence of categories
\[
\rlscomod C\sim\rdlsmod{\sd C}.
\]
Since CT- and NF-spaces are reflexive, one can prove similarly to
proposition \ref{prop:DualNFisCT} that dual of NF-$\cten$-module
is a CT-$\cten$-comodule. Thus we also have an anti-equivalence of
categories 
\[
\rctcomod C\sim\rnfmod{\sd C}.
\]

\begin{defn}
On $V$ there is also a left topological $\sd C-$$\cten$-module
structure
\[
\begin{array}{ccc}
m:\mbox{ }\sd C\cten V\to V\\
\lambda\otimes v\mapsto\lambda\cdot v & = & \left(id_{V}\bar{\otimes}\lambda\right)\circ\rho_{V}\left(v\right).
\end{array}
\]
Such $\cten$-modules (with $\cten$-module structure coming from
$\cten$-comodule structure) are called \emph{induced}.
\end{defn}

\subsection{Rationality}

Similar to the section \ref{sub:Rational-modules}, one can define
a notion of a \emph{t-rational }left LS-$\cten$-module over $\sd C$.
If $V$ is a left $\sd C-\cten-$module and a compact type LCTVS,
then, by proposition \ref{prop:topten}.\ref{enu:int}, we have the
map \ref{eq:ratmod} $\rho_{V}:V\to V\cten C\cong L_{b}\left(\sd C,V\right)$
(where $\rho_{V}\left(v\right)\left(\lambda\right)=\lambda\cdot v),$
such that the $\sd C$-module structure on $V$ is exactly $\lambda\cdot v=\left(\id V\otimes\lambda\right)\circ\rho_{V}\left(v\right).$
The $\cten$-module axioms for $V$ imply that $\rho_{V}$ satisfy
right $\cten$-comodule axioms (\cite[2.1.1]{sw}). Thus we have 
\begin{prop}
All continuous $\sd C$-$\cten$-modules structures on LS-space are
t-rational.
\end{prop}
Similarly to the Banach case, morphisms of t-rational modules are
comodule morphisms. Thus the above proposition gives an equivalence
of categories
\[
\rlscomod C\sim_{\sd C}\mbox{LSMod}.
\]

\subsection{Quotients, subobjects and simplicity}
\begin{prop}
\label{pro:Let-V-be-1}Let V be a right (CT-) LS-$\cten$-comodule
over C and U be a closed $\cten$-subcomodule of C. Then U and $V/U$
are right (CT-) LS-$\cten$-comodules over C and the exact sequence
\[
0\to U\to V\to V/U\to0
\]
give rise to the exact sequence of right (NF-) FS-$\cten$-modules
over $\sd C$ 
\[
0\to\left(V/U\right)'_{b}\to V_{b}'\to U_{b}'\to0
\]
with the morphisms in this exact sequence being strict.\end{prop}
\begin{proof}
Similar to \cite[1.2]{ST4}.\end{proof}
\begin{prop}
\label{prop:hdcom}Let $M\in\rnfmod A$ and $\left\{ M_{n}\right\} $
is an $\left\{ A_{n}\right\} $-structure on M. Then each $M_{n}$
is a topological A-$\cten$-module and
\begin{enumerate}
\item Each $\sd{M_{n}}$ is an LB $\sd A$-$\cten$-comodule;
\item $\bhd{M_{n}}\cong\sd{M_{n}}\ccten{\sd A}\bhd{A_{n}}$. 
\end{enumerate}
\end{prop}
\begin{proof}
Since $M_{n}$ is an $A$-module via $\pi_{n}:A\to A_{n}$ with the
module structure given by the map $m_{M_{n}}\circ\left(\id{M_{n}}\otimes\pi_{n}\right):M_{n}\cten A\to M_{n}$,
it follows from proposition \ref{prop:topten}.\ref{enu:-and-CTBten}
that taking duals gives us the map $\left(m_{M_{n}}\circ\left(\id{M_{n}}\otimes\pi_{n}\right)\right)'=\cm_{n}:\sd{M_{n}}\to\sd{M_{n}}\cten\sd A$.
$\cm_{n}$ satisfies $\cten$-comodule axioms, which proves the first
statement. The image of $\sd{M_{n}}\ccten{\sd A}\bhd{A_{n}}$ in $\sd{M_{n}}$
under $\id{\sd{M_{n}}}\bten\epsilon_{A}$ is the closure of all elements
$m$ such that $\rho_{\sd{M_{n}}}\left(m\right)\in\sd{M_{n}}\cten\bhd{A_{n}}$,
which is exactly the definition of $\bhd{M_{n}}$.\end{proof}
\begin{defn}
We call a topological $\cten$-comodule $V$ simple, if it does not
have any closed $\cten$-subcomodules. \end{defn}
\begin{prop}
Let $V\in\rlscomod C$. Then 
\begin{enumerate}
\item V is simple if and only if $V'_{b}$ is simple as $\bd C-$$\cten$-module;
\item V is simple if and only if V is simple as $\bd C$-$\cten$-module.
\end{enumerate}
\end{prop}
\begin{proof}
(1) follows from Proposition \ref{pro:Let-V-be-1}. (2) is proved
similarly to Banach case.
\end{proof}
The following is the analog of \cite[lemma 3.9]{ST3}
\begin{prop}
\label{pro:Let-V-be}Let V be a $\left(C_{n}\right)$-$\cten$-comodule
over a CT-$\cten$-coalgebra $C$ and let $\rho_{V}|_{V_{n}}=\rho_{V_{n+1}}|_{V_{n}}=\rho_{V_{n}}.$
Suppose there exists N>0 such that $V_{n}$ are simple for all n>N.
Then V is simple.\end{prop}
\begin{proof}
Let $U\subset V$ be a proper closed $\cten$-subcomodule. Since $V$
is of compact type, $U=\ilim U_{n},$ such that $U_{n}$ are closed
subspaces of $V_{n}.$ Since $U$ and $V_{n}$ are $\cten$-subcomodules
of $V,$ $U_{n}$ are $C_{n}-$$\cten$-subcomodules of $V_{n}.$
Since $U$ is a proper subspace of $V,$ $U_{n}$ must be proper subcomodules
of $V_{n}$ for all $n>N$ for some $N>0$, and this is a contradiction
with the simplicity of $V_{n}$.\end{proof}
\begin{rem}
\cite[lemma 3.9]{ST3} is proved for coadmissible $\cten$-modules
over Fréchet-Stein algebra, i.e. the duals of $V_{n}$ are supposed
to be finitely generated over $\bd{C_{n}}$ and $\bd{C_{n}}$ are
supposed to be Noetherian. These assumptions are not required in our
result. On the other hand, our result is for CT-$\cten$-comodules,
which on the dual side mean a nuclear Fréchet $\cten$-module, and
nuclearity is not required in \cite[lemma 3.9]{ST3}.
\end{rem}

\subsection{Cotensor product}

Let $A$ and $B$ be LS-$\cten$-coalgebras. Let $\left(M,\rho_{M}\right)$
be a right LS $A-$$\cten$-comodule and let $\left(N,{}_{N}\rho\right)$
be a left LS $A$-$\cten$-comodule. Similar to the Banach case, one
can give the following definition.
\begin{defn}
The space $M\ccten AN=Ker\left(\rho_{M}\otimes id_{N}-id_{M}\otimes{}_{N}\rho\right)$
is called cotensor product of $M$ and $N$ over $A.$ 

Since $M\ccten AN$ is a kernel of a continuous map, it is a closed
subspace of $M\cten N$, which is a LS-space, and thus is a LS-space.
It is an equalizer of $\rho_{M}\otimes id_{N}$ and $id_{M}\otimes{}_{N}\rho$.

If $N$ is also a right LS $B$-$\cten$-comodule, then $M\ccten AN$
is also a right LS $B$-$\cten$-comodule.

Since under antiequivalences of categories equalizers and coequalizers
are dual to each other, we have the following result.\end{defn}
\begin{prop}
\label{lem:CTcotenten}$\sd{M\ccten AN}=\sd M\underset{\sd A}{\cten}\sd N$,
$\sd{\sd M\underset{\sd A}{\cten}\sd N}=M\ccten AN$.
\end{prop}
Similarly one defines induction and restriction functors. Lemma \ref{lem:-is-a}
and Proposition \ref{pro:(Frobenius-reciprocity)} (Frobenius reciprocity)
remain true in LS case.

If $M\simeq\ilim M_{n}$ and $N\simeq\ilim N_{n}$ are CT-$\cten$-comodules,
then $M\cten N=\ilim M_{n}\cten N_{n}$. Since $\rho_{M}$ and $\rho_{N}$
are defined by $\rho_{M_{n}}$ and $\rho_{N_{n}}$, we have $M\ccten AN=\ilim M_{n}\ccten{A_{n}}N_{n}$. 

Similar to the case of Banach $\cten$-comodules, one defines induction
and restriction functors. Frobenius reciprocity and Tensor identity
remain true for CT-$\cten$-comodules.

\section{Admissibility}
\begin{defn}
Let $A=\plim A_{n}$ be an NF-$\cten$-algebra. An NF $\left\{ A_{n}\right\} $-module
$M=\plim M_{n}$ is called $\left\{ A_{n}\right\} $-coadmissible
if 
\begin{enumerate}
\item Each $M_{n}$ is finitely generated;
\item For all $n$ we have a topological isomorphism $M_{n}\cong M\cten_{A}A_{n}$.
\end{enumerate}
\end{defn}
\begin{rem}
Recall the definitions from \cite{EM}. A weak Fréchet-Stein $\cten$-algebra
is a projective limit of locally convex hereditary complete topological
$\cten$-algebras, s.t. transition maps factor through Banach spaces
(BH-maps), thus this notion is slightly more general than our notion
of Fréchet $\cten$-algebra.

A module $M$ over a weak Fréchet-Stein $\cten$-algebra $A$ is called
coadmissible (w.r.t. a fixed nuclear Fréchet structure $A_{n}$ on
$A$) \cite[1.2.8]{EM}, if we have the following data:
\begin{enumerate}
\item a sequence of finitely generated Banach $\cten$-modules $M_{n}$
over $A_{n}$;
\item an isomorphism of Banach $A_{n}$-$\cten$-modules $M_{n+1}\underset{A_{n+1}}{\cten}A_{n}\simeq M_{n};$
\item an isomorphism of topological $A$-$\cten$-modules $M\simeq{\displaystyle \plim M_{n}}$
(projective limit is taken w.r.t. transition maps $M_{n+1}\longrightarrow M_{n},$
induced by (2)).
\end{enumerate}
Condition (3) is already contained in our definition of NF-$\cten$-module.
It is easy to see that our condition (2) imply condition (2) as in
loc. cit. \cite[Cor. 3.1]{ST3} states that over a Fréchet-Stein algebra
$A$ there is an isomorphism $M\underset{A}{\ten}A_{n}\simeq M_{n}$.
Since $M_{n+1}\underset{A_{n+1}}{\cten}A_{n}\simeq M_{n+1}\underset{A_{n+1}}{\ten}A_{n}$,
one can see that $M\underset{A}{\ten}A_{n}\simeq M_{n}$ is equivalent
to the topological isomorphism $M\underset{A}{\cten}A_{n}\simeq M_{n}$
(see also (the proof of) \cite[1.2.11]{EM}). Thus over Fréchet-Stein
$\cten$-algebra these two conditions are equivalent and these two
definitions coincide. In general, our condition (2) is stronger and
results in existence of a duality between $\cten$-modules and $\cten$-comodules.
\end{rem}
It is known that if $M$ is coadmissible w.r.t. one nuclear Fréchet
structure on $A,$ then it is coadmissible w.r.t. any \cite[1.2.9]{EM}. 

We want to study an analog of this notion for $\cten$-comodules.
\begin{defn}
Let $C$ be a CT-coalgebra with CT-structure $\left\{ C_{n}\right\} $.
Let $V$ be a right $\left\{ C_{n}\right\} $-comodule with CT-structure
$\left\{ V_{n}\right\} _{n\in\mathbb{N}}$. Suppose that
\begin{enumerate}
\item $V_{n}$ is a finitely cogenerated Banach right $\cten$-comodule
over $C_{n}$;
\item we have a topological isomorphism of $C_{n}$-$\cten$-comodules $V\ccten CC_{n}\simeq V_{n}$
($C_{n}$ is a left and right $C-$$\cten$-comodule, so we can take
completed cotensor product).
\end{enumerate}
We call $V$ \emph{admissible} w.r.t. CT-structure $\left\{ \left(C_{n},\phi_{n,m}\right)\right\} _{n\in\mathbb{N}}$
(or $\left\{ C_{n}\right\} $-\emph{admissible}). We call $\left\{ V_{n}\right\} _{n\in\mathbb{N}}$
a $\left\{ C_{n}\right\} $\emph{-admissible structure}.
\end{defn}
The second condition essentially means that the only elements $x\in V$,
such that $\rho_{V}\left(x\right)\in V\cten C_{n}$, are the elements
of $V_{n}$.
\begin{prop}
\label{prop:IndAdm}Let $\phi:A\to B$ be a morphism of CT-$\cten$-coalgebras,
given by a morphism of CT-structures $\left\{ \phi_{n}:A_{n}\to B_{n}\right\} $
with CT-structure $\left\{ A_{n}\right\} _{n\in\mathbb{N}}$ for A
and CT-structure $\left\{ B_{n}\right\} _{n\in\mathbb{N}}$ for B.
Then:
\begin{enumerate}
\item \label{enu:propIndAdm1}If $W$ is $\left\{ B_{n}\right\} _{n\in\mathbb{N}}$-admissible
B-comodule, then $W^{\phi}$ is an $\left\{ A_{n}\right\} _{n\in\mathbb{N}}$
-admissible A-comodule;
\item If V is $\left\{ A_{n}\right\} _{n\in\mathbb{N}}$ -admissible A-comodule,
then $V_{\phi}$ is an $\left\{ B_{n}\right\} _{n\in\mathbb{N}}$-admissible
B-comodule.
\end{enumerate}
\end{prop}
\begin{proof}
In both cases we have to check the three conditions of admissibility.
\begin{itemize}
\item $W^{\phi}$ case: 

First note that $W^{\phi}=W\ccten B{}_{\phi}A$ is a closed subspace of the LS-space $W\cten A$ and thus is the LS-space itself. So we have a topological isomorphism $W\ccten B{}_{\phi}A\cong\ilim W_{n}\ccten{B_{n}}{}_{\phi_n}A_{n}\cong\ilim{W_{n}^{\phi_n}} $ and thus $W^{\phi}$ is a CT-comodule w.r.t. $A_{n}$. Lets prove that ${W_{n}^{\phi_n}}$ is $\left\{A_{n}\right\}$-admissible structure.
\begin{enumerate}
\item follows from Proposition \ref{prop:IndFinCogen}.
\item First note that $W_n\ccten {B_n}{}\left(_{\phi_n}A_n\right)\cong W\ccten {B}{}\left(_{\phi}A_n\right)$. Indeed, since $W$ is $\left\{ B_{n}\right\} _{n\in\mathbb{N}}$-admissible, we have the above isomorphism as equality of sets. \cite[11.1.2]{PGSch}
imply that topologies also coinside.

Now we have
\[
W^{\phi}\ccten AA_{n}\cong \left(W\ccten B{}_{\phi}A\right)\ccten AA_{n}\cong \left(W\ccten B{}_{\phi}A\right)\ccten {A_n}A_{n}\cong \left(W\ccten B{}_{\phi}A_n\right)\ccten {A_n}A_{n}\cong
\]
\[
\cong {W_n^{\phi_n}}\ccten A{}A_{n}\cong {W_n^{\phi_n}}\ccten {A_n}A_{n}\cong W_{n}^{\phi_n}.
\]

\end{enumerate}
\item $V_{\phi}$ case:

The fact that $V_{\phi}$ is a CT-comodule is obvious.

\begin{enumerate}
\item if $V_{n}$ is embedded into $A_{n}^{k}$ and $A_{n}$ is embedded
into $B_{n}^{m},$ then we have an embedding $V_{n}\hookrightarrow B_{n}^{km}.$
\item we have $V\ccten AA_{n}\cong V_{n}$. Then 
\[
\left(V\right)_{\phi}\ccten BB_{n}\cong\left(V\ccten AA\right)_{\phi}\ccten BB_{n}\cong\left(V\ccten AA_{\phi}\right)\ccten BB_{n}\cong\left(V\ccten AA_{\phi}\right)\ccten {B_n}B_{n}\cong
\]
\[
\cong\left(V\ccten A\left(A_n\right)_{\phi_n}\right)
\cong\left(V\ccten AA_n\right)_{\phi_n}\cong\left(V_{n}\right)_{\phi_n}.
\]
\end{enumerate}
\end{itemize}
\end{proof}
\begin{cor}
If a CT-$\cten$-comodule is admissible w.r.t. one CT-structure, then
it is admissible w.r.t. any.\end{cor}
\begin{proof}
Similar to \cite[1.2.9]{EM}. Let $\left\{ A_{n}\right\} _{n\in\mathbb{N}}$
and $\left\{ B_{n}\right\} _{n\in\mathbb{N}}$ be two different CT-structures
on a CT-$\cten$-coalgebra $A$ and let $M$ be $\left\{ B_{n}\right\} $-admissible
with admissible structure $\left\{ M_{n}\right\} $. Let $\psi:\mathbb{N}\to\mathbb{N}$
be the increasing map giving the morphisms $\Psi:\left\{ A_{n}\right\} \to\left\{ B_{n}\right\} $,
$\Psi\left(n\right):A_{n}\to B_{\psi\left(n\right)}$. Then we have
an automorphism $\Psi:A\to A$, the system $\left\{ B_{\psi\left(n\right)}\right\} $
is also a CT-structure, $M$ is $\left\{ B_{\psi\left(n\right)}\right\} $-admissible
and $M^{\Psi}\simeq M$ is $\left\{ A_{n}\right\} $-admissible by
Proposition \ref{prop:IndAdm}.\ref{enu:propIndAdm1}.
\end{proof}
Next we show that coadmissible $\cten$-modules and admissible $\cten$-comodules
are dual to each other.
\begin{thm}
Let $C\cong\ilim C_{n}$ be a CT-$\cten$-coalgebra. If $\left\{ V_{n}\right\} $
is an $\left\{ C_{n}\right\} $-admissible structure for CT-$\cten$-comodule
$V$ over $C$, then $\sd{V_{n}}$ is a $\left\{ \sd{C_{n}}\right\} $-coadmissible
structure for the NF-module $\sd V$ over $\sd C$.\end{thm}
\begin{proof}
Let $\left\{ V_{n}\right\} $ be $\left\{ C_{n}\right\} $-admissible
structure on $V$. We have $V_{n}\cong V\ccten CC_{n}$, which is
by definition the kernel\begin{center}
\begin{tikzpicture}[auto]    
\node (o) {$0$};   
\node (c) [right= and 0.3cm of o] {$V\ccten C C_{n}$};   
\node (cc) [right= and 0.7cm of c] {$V\cten C_{n}$};   
\node (ccc) [right= and 0.7cm of cc] {$V\cten C\cten C_{n}$,};   
\draw[->] (o) to node {} (c);
\draw[->] (c) to node {$\phi$} (cc);
\draw[->] (cc) to node {$\psi$} (ccc);
\end{tikzpicture} 
\end{center}where $\psi=\rho_{V}\cten\id{C_{n}}-\id V\cten\rho_{C_{n}}$.

Taking duals gives us the following sequence\begin{center}
\begin{tikzpicture}[auto]    
\node (o) {$0$};   
\node (c) [right= and 0.3cm of o] {$\sd{V\ccten C C_{n}}$};   
\node (cc) [right= and 0.6cm of c] {$\sd{V\cten C_{n}}$};   
\node (ccc) [right= and 0.6cm of cc] {$\sd{V\cten C\cten C_{n}}$.};   
\draw[->] (c) to node[swap] {} (o);
\draw[->] (cc) to node[swap] {$\phi'$} (c);
\draw[->] (ccc) to node[swap] {$\psi'$} (cc);
\end{tikzpicture} 
\end{center}

By proposition \ref{prop:topten}.\ref{enu:-and-CTBten} this sequence
is equal to \begin{center}
\begin{tikzpicture}[auto]    
\node (o) {$0$};   
\node (c) [right= and 0.3cm of o] {$\sd{V\ccten C C_{n}}$};   
\node (cc) [right= and 0.5cm of c] {$\sd{V}\cten \sd{C_{n}}$};   
\node (ccc) [right= and 0.5cm of cc] {$\sd{V}\cten \sd{C}\cten \sd{C_{n}}$};   
\draw[->] (c) to node[swap] {} (o);
\draw[->] (cc) to node[swap] {$\phi'$} (c);
\draw[->] (ccc) to node[swap] {$\psi'$} (cc); 
%%  $} (cc);
\end{tikzpicture} 
\end{center}and $\psi'=m_{\sd V}\cten\id{\bd{C_{n}}}-\id{\sd V}\cten m_{\bd{C_{n}}}$.

Since $\phi$ is an embedding, $\bd{\phi}$ is surjective and its
kernel is a {*}-weak closure of the elements of the form $v'c'\otimes c_{n}'-v'\otimes c'c_{n}'$.
Thus we have a linear surjection $\sd V\cten_{\sd C}\sd{C_{n}}\overset{\sim}{\longrightarrow}\sd{V\ccten CC_{n}}$,
which is continuous by the universal property of the complete tensor
product. The embeddings $\bd{V_{m}}\cten_{\bd{C_{m}}}\bd{C_{n}}\emb\bd{\left(V_{m}\ccten{C_{m}}C_{n}\right)}$,
from proposition \ref{prop:ctenemb}, give rise to the embedding $\sd V\cten_{\sd C}\sd{C_{n}}\emb\sd{V\ccten CC_{n}}$
and thus the above surjection is a continuous isomorphism.

Since $\sd{V\ccten CC_{n}}\cong\sd{V_{n}}$ is, in fact, a Banach
space and $\sd V\cten_{\sd C}\sd{C_{n}}$ is a quotient of Fréchet
space $\sd V\cten\sd{C_{n}}$, by Open Mapping theorem \ref{thm:(Open-Mapping-theorem)},
we have the following topological isomorphism 
\[
\sd{V_{n}}\cong\sd{V\ccten CC_{n}}\cong\sd V\cten_{\sd C}\sd{C_{n}},
\]
which is the 2nd axiom of coadmissibility. The check of the axiom
1 is trivial.
\end{proof}
Before proving the inverse statement we prove an auxiliary result.
\begin{lem}
\label{lem:lemop}Let $f:V\to U$ be an open surjective morphism of
LCTVS. Then we have a continuous isomorphism $\sd U\simeq\ann{\ker f}$.\end{lem}
\begin{proof}
We have a dual map $\bd f:\sd U\to\sd V$, which is continuous and,
clearly, $im\left(\bd f\right)\subset\ann{\ker f}$. Consider a map
$\bar{f}:\ann{\ker f}\to\sd U$ defined as $\bar{f}\left(\phi\right)\left(u\right)=\phi\left(v\right)$
for $u=f\left(v\right)$. If $u=f\left(v'\right)$, then $v-v'\in\ker f$
and $\bar{f}\left(\phi\right)\left(0\right)=\phi\left(v-v'\right)=0$.
Thus $\bar{f}\left(\phi\right)$ and the map $\bar{f}$ are well defined.
For an open set $U\subset K$, $\bar{f}\left(\phi\right)^{-1}\left(U\right)=f\left(\phi^{-1}\left(U\right)\right)$
is open, since $\phi$ is continuous and $f$ is open. Thus $\bar{f}\left(\phi\right)\in\sd U$.
Clearly $\bar{f}$ and $\bd f$ are inverse to each other and thus
$\bd f$ is a bijection onto $\ann{\ker f}$.\end{proof}
\begin{thm}
Let $A\cong\plim A_{n}$ be an NF-$\cten$-algebra and $M\cong\plim M_{n}$
is a coadmissible A-$\cten$-module with $\left\{ A_{n}\right\} $-coadmissible
structure $\left\{ M_{n}\right\} $. Then $\sd M$ is an admissible
CT-$\cten$-comodule over the CT-$\cten$-coalgebra $\sd A$ with
$\left\{ \bhd{A_{n}}\right\} $-admissible structure $\left\{ \bhd{M_{n}}\right\} $.\end{thm}
\begin{proof}
Consider the defining exact sequence of Fréchet spaces for $M\cten_{A}A_{n}\cong M_{n}$,
\begin{center}
\begin{tikzpicture}[auto]    
\node (o) {$0$.};   
\node (c) [left= and 0.3cm of o] {$M\cten_A A_{n}$};   
\node (cc) [left= and 1cm of c] {$M\cten A_{n}$};   
\node (ccc) [left= and 1cm of cc] {$M\cten A\cten A_{n}$};   
\draw[->] (c) to node {} (o);
\draw[->] (cc) to node {$\phi$} (c);
\draw[->] (ccc) to node {$\psi$} (cc);
\end{tikzpicture} 
\end{center}where $\psi=m_{M}\cten\id{A_{n}}-\id M\cten m_{A_{n}}$. Taking duals
and applying proposition \ref{prop:topten}.\ref{enu:-and-CTBten}
gives us the sequence \begin{center}
\begin{tikzpicture}[auto]    
\node (o) {$0$.};   
\node (c) [left= and 0.3cm of o] {$\sd{M\cten_A A_{n}}$};   
\node (cc) [left= and 0.5cm of c] {$M_b'\cten \sd{A_{n}}$};   
\node (ccc) [left= and 0.5cm of cc] {$M_b'\cten A_b'\cten \sd{A_{n}}$};   
\draw[->] (o) to node {} (c);
\draw[->] (c) to node[swap] {$\phi'$} (cc);
\draw[->] (cc) to node[swap] {$\psi'$} (ccc);
\end{tikzpicture} 
\end{center}with $\psi'=\rho_{M_{b}'}\cten\id{\sd{A_{n}}}-\id{M_{b}'}\cten\rho_{\sd{A_{n}}}$.
By definition, $M\cten_{A}A_{n}=M\cten A_{n}/\compl I$, where $\compl I$
is the closure of the linear hull of the set of elements of the form
$mc\otimes a-m\otimes ca$. By lemma \ref{lem:lemop}, $\sd{M\cten_{A}A_{n}}\simeq\ann{\compl I}$.
For any linear map $f:V\to U$ we have $\ker\bd f=\ann{im\left(f\right)}=\ann{\compl{im\left(f\right)}}$.
Since $\compl I$ is the closure of the image of the map $m_{M}\cten\id{A_{n}}-\id M\cten m_{A_{n}}:M\cten A\cten A_{n}\to M\cten A_{n}$,
we have $\sd{M\cten_{A}A_{n}}\simeq\ann{\compl I}=\ker\bd{\left(m_{M}\cten\id{A_{n}}-\id M\cten m_{A_{n}}\right)}$
and since 
\[
\bd{\left(m_{M}\cten\id{A_{n}}-\id M\cten m_{A_{n}}\right)}=\rho_{\sd M}\cten\id{\sd{A_{n}}}-\id{\sd M}\cten\rho_{\sd{A_{n}}},
\]
we get a continuous isomorphism $\sd{M\cten_{A}A_{n}}\simeq\sd M\ccten{\sd A}\sd{A_{n}}$.
Since $\sd{M_{n}}\cong\sd{M\cten_{A}A_{n}}$ is a Banach space and
$\sd M\cten\sd{A_{n}}$ is an LB-space (in fact, regular LB-space),
$\bd{\phi}\left(\sd{M\cten_{A}A_{n}}\right)$ is embedded into $\sd{M_{k}}\cten\sd{A_{n}}$
for some $k$. Thus both $\sd{M\cten_{A}A_{n}}$ and $\sd M\ccten{\sd A}\sd{A_{n}}$
are Banach spaces and by Open Mapping theorem \ref{thm:(Open-Mapping-theorem)},
we have the topological isomorphism $\sd{M\cten_{A}A_{n}}\cong\sd M\ccten{\sd A}\sd{A_{n}}$.
Note that, since $A_{n}$ is a Banach $A-A$-$\cten$-bimodule (via
projective limit structure map $\psi_{n}:A\to A_{n}$), $\sd{A_{n}}$
is a Banach $\sd A-\sd A$-$\cten$-bicomodule by proposition \ref{prop:topten}.\ref{enu:-and-CTBten}.

To prove the claim, first we remark that $\bhd{A_{n}}$ is also a
Banach $\sd A-\sd A$-$\cten$-bicomodule, since $\sd A=\ilim\bhd{A_{n}}$.
Consider $N_{n}=\sd M\ccten{\sd A}\bhd{A_{n}}$. Each $\sd M\ccten{\sd A}\bhd{A_{n}}$
is a closed subspace of $\sd M\ccten{\sd A}\sd{A_{n}}$, and thus
$N_{n}$ is a closed subspace of $\sd{M\cten_{A}A_{n}}\cong\sd{M_{n}}$.
We have $N_{n}=\sd M\ccten{\sd A}\bhd{A_{n}}\cong\sd M\ccten{\sd A}\sd{A_{n}}\ccten{\sd{A_{n}}}\bhd{A_{n}}\cong\sd{M_{n}}\ccten{\sd{A_{n}}}\bhd{A_{n}}$
and by proposition \ref{prop:hdcom} $N_{n}\cong\bhd{M_{n}}$. Since
$\bhd{\left(-\right)}$ preserves finite direct sums, $\bhd{M_{n}}$
is a finitely cogenerated $\bhd{A_{n}}$-$\cten$-comodule, which
gives us the 1st axiom of admissibility. The 2nd axiom is in the definition
of $N_{n}$.\end{proof}
\begin{cor}
Let $C$ be a CT-$\cten$-coalgebra and $\sd C$ its dual NF-$\cten$-algebra.
Then the categories of admissible $C$-$\cten$-comodules and coadmissible
$\sd C$-$\cten$-modules are antiequivalent. In particular, if $\sd C$
is Fréchet-Stein, then by \cite[3.5]{ST3} the category of admissible
$C$-$\cten$-comodules is abelian.\end{cor}
\begin{rem}
One can define admissible $\cten$-comodules in analogy with the definition
\cite[1.2.8]{EM}, i.e. to require that
\begin{enumerate}
\item $V_{n}$ is a finitely cogenerated Banach right $\cten$-comodule
over $C_{n}$;
\item we have an isomorphism of Banach $C_{n}$-$\cten$-comodules $V_{n+1}\ccten{C_{n+1}}C_{n}\simeq V_{n}$
($C_{n}$ is a left and right $C_{n+1}-$$\cten$-comodule, so we
can take completed cotensor product).
\end{enumerate}
With this definition admissibility will still be preserved by induction
and restriction functors and will not depend on CT-structure. However
for duality ``admissible $\cten$-comodules'' - ``coadmissible
$\cten$-modules'' one needs the subspaces $V\ccten CC_{n}\subset V\cten C_{n}$
to be (ultra)bornological. For general admissible $\cten$-comodule
in the sense of this remark, it is currently unknown to the author
if it is always true.
\end{rem}

\section*{Appendix}

Here we prove two technical results, stated in the proposition \ref{prop:topten}.
\begin{lem}
\label{lem:LBsplitting}Let V be an LS-space and $W$ be a Banach
space. Then for any bounded set $B$ in $V\cten W$ there exists a
bounded set $B_{V}\subset V$ and a bounded set $B_{W}\subset W$
such that $B\subset\compl[\left(V\cten W\right)]{\left(B_{V}\otimes B_{W}\right)}$.\end{lem}
\begin{proof}
Let $0<t<1$. Similar to the proof of \cite[10.4.6]{PGSch}, for any
element $z\in V\cten W$ one has a presentation $z={\displaystyle \sum_{k=0}^{\infty}}v_{k}^{z}\ten w_{k}$
with $\left\{ w_{k}\right\} \in W\backslash\left\{ 0\right\} $ being
a $t$-orthogonal base in $W$ (follows from \ref{prop:topten}.\ref{enu:closedLB}
and \cite[4.30.ii]{vanr}). The rest of the proof goes the same way
as in \cite[10.4.6]{PGSch}.\end{proof}
\begin{rem*}
The result is true for more general $V$. However the proof is more
complicated, since instead of sums (or sequences) one has to work
with nets. The above version is sufficient for our purposes.
\end{rem*}
\smallskip{}

\begin{rem*}
One can make $B_{V}$ and $B_{W}$ into $K^{0}$-submodules by considering
their convex hulls.
\end{rem*}
We will use the notation $B_{V}\cten B_{W}:=\compl[\left(V\cten W\right)]{\left(B_{V}\otimes B_{W}\right)}$.

\subsubsection*{Proof of proposition \ref{prop:topten}.\ref{enu:-and-CTBten}.}

We first note that, since $V$ is bornological and reflexive \cite[1.1]{ST4}
and $W$ is Banach space, by \cite[18.8]{nfa} we have a topological
isomorphism 
\[
\sd V\cten\sd W\cong L_{b}\left(V,\sd W\right).
\]
Since $V$ is reflexive, by \cite[15.5]{nfa} it is barrelled and
by \cite[1.1.35]{EM} (via putting $W=K$ there) we have a continuous
bijection 
\[
\phi:\sd{V\cten W}\simeq L_{b}\left(V,\sd W\right)
\]
and, similar to the proof of \cite[20.13]{nfa}, a correspondence
of open lattices 
\[
\phi:L\left(l_{V}\cten l_{W},K^{0}\right)\mapsto L\left(l_{V},L\left(l_{W},K^{0}\right)\right),
\]
where $l_{V}$ and $l_{W}$ are bounded $K^{0}$-submodules in $V$
and $W$. The open lattices $L\left(l_{V},L\left(l_{W},K^{0}\right)\right)$
generate the strong topology on $L_{b}\left(V,\sd W\right)$ by definition.
Our splitting lemma \ref{lem:LBsplitting} imply that the collection
of bounded sets of the form $\left\{ l_{V}\cten l_{W}\right\} $ and
of all bounded sets in $V\cten W$ satisfy the assumptions from \cite[6.5]{nfa}
for families $\mathcal{B}$ and $\tilde{\mathcal{B}}$ correspondingly.
Thus \cite[6.5, 6.4]{nfa} imply that open lattices $L\left(l_{V}\cten l_{W},K^{0}\right)$
generate the strong topology on $\sd{V\cten W}$ and thus $\phi$
is a topological isomorphism.

The second identity is \cite[20.13]{nfa}.\qed

\smallskip{}

Recall that an LCTVS $V$ is called \emph{bornological} if its topology $\tau_V$ is the
final topology for a system (not necessarily a sequence) of maps $\{E_i\to E\}_{i\in I}$
from normed vector spaces $E_i$. It is called \emph{ultrabornological} if
the spaces $E_i$ are Banach spaces.
Clearly a complete bornological space is ultrabornological and locally convex inductive 
limits of Banach spaces are ultrabornological. Our reference
for (ultra)bornological spaces is \cite[Chapter 24]{MEI}, with nonarchimedean
analogs of most statements there being easy exercises (also see \cite{nfa}).
For any unfamiliar term in the following proof one should follow the
references in \cite[Chapter 24]{MEI}.

\subsubsection*{Proof of proposition \ref{prop:topten}.\ref{enu:closedLB}.}

Let $Z$ be a Banach space and $f:V\cten W\to Z$ be a locally bounded
map and $\phi:V\times W\to V\cten W$ be the canonical bilinear map. 

For every $v\in V$ consider the map $f\circ\phi\left(v,-\right):W\to Z$.
For every bounded subset $B\subset W$ the set $v\otimes W$ is bounded
in $V\cten W$ and thus $f\circ\phi\left(v,-\right)$ is a bounded
map of Banach spaces, which imply it is continuous for every $v\in V$.
Similarly, for every $w\in W$, $f\circ\phi\left(-,w\right):V\to Z$
is a locally bounded map, which imply it is continuous since $V$,
as an LS-space, is ultrabornological \cite[24.10.4]{MEI}. Thus $f\circ\phi$
is a separately continuous map. By the universal property of inductive
tensor product, $f$ is continuous.

Thus we get that each locally bounded map $f$ is continuous, which
imply that $V\cten W$ is bornological \cite[24.10.4]{MEI}. Since
it is also complete, it is ultrabornological \cite[24.15.b]{MEI}.
$\ilim\left(V_{n}\cten W\right)$ is a webbed space as a locally convex
inductive limit of webbed spaces and thus by Open Mapping theorem
\cite[24.30]{MEI} the canonical continuous bijection $\ilim\left(V_{n}\cten W\right)\simeq V\cten W$
is open, which imply it is a topological isomorphism.\qed

\end{document}